\documentclass[a4paper, 11pt]{article}

\usepackage[francais,english]{babel}
\usepackage[latin1]{inputenc}
\usepackage[T1]{fontenc}
\usepackage{ae,aecompl,aeguill}
\usepackage[dvips,final]{graphicx}
\usepackage{amssymb,amsfonts,amsthm,amsmath}
\usepackage{geometry}
\usepackage{multirow}

\usepackage{babelbib}
\newtheorem{thm}{Théorème}[section]
\newtheorem{lem}[thm]{Lemme}
\newtheorem{prop}[thm]{Proposition}
\newtheorem{coro}[thm]{Corollaire}
\newtheorem{ex}[thm]{Exemple}
\newtheorem{rem}[thm]{Remarque}

\def\matg{{\left(\begin{smallmatrix} a & b\\c & d\end{smallmatrix}\right)}}
\newcommand\mat[4]{{\left(\begin{smallmatrix} #1 & #2\\#3 & #4\end{smallmatrix}\right)}}

\def\Z{{\mathbb{Z}}}
\def\C{{\mathbb{C}}}
\def\R{{\mathbb{R}}}
\def\Q{{\mathbb{Q}}}

\def\pte{{\mathbb{P}^1(\mathbb{Q})}}
\def\H{{\mathbb{H}}}
\def\G{{SL_2(\Z)}}

\renewcommand{\Re}{\mathrm{Re}}
\renewcommand{\d}{\, \mathrm{d}}
\newcommand{\ra}{\Rightarrow}

\def\Hom{{\mathrm{Hom}}}
\renewcommand{\dim}{\mathrm{dim}}
\renewcommand{\Im}{\mathrm{Im}}
\def\Ker{{\mathrm{Ker}}}
\def\Per{{\mathrm{Per}}}
\def\tr{{\mathrm{tr}}}
\def\harm{{\mathrm{harm}}}

\title{Relations de Manin d'ordre $2$ vérifiées par le polynôme des bi-périodes d'un couple de formes modulaires
}
\author{Nicolas Provost
\protect\footnote{Cet article fait suite à un travail de recherche d'étude doctorale effectué sous la direction de Loïc Merel à l'université Paris VII - Denis Diderot.}
}
\begin{document}

\selectlanguage{francais}

\maketitle

\abstract{
Les relations de Manin caractérisent principalement les polynômes des périodes d'une forme modulaire. Nous proposons dans cet article un nombre fini de relations d'ordre $2$ vérifiées par le polynôme des bi-périodes d'un couple de formes. En construisant une partie Eisenstein d'ordre $2$, nous précisons l'écart entre ces polynômes des bi-périodes et ceux annulés par les relations établies. En donnant une description calculatoire de cette partie, nous démontrons un résultat d'irrationalité du quotient des périodes pour certains poids.
}

\selectlanguage{english}

\abstract{
Manin's relations characterize period polynomials of a modular form. In this paper, we propose a finite set of relations checked by bi-period polynomials of a couple of forms. Then we construct an Eisenstein part of depth $2$ which essentially completed the set of bi-periods polynomials among polynomials canceled by the relations. By giving a computational description of that part, we prove a result of irrationality of ratio of the periods for some weights.
}

\selectlanguage{francais}


\section{Introduction}

Manin \cite{Ma1} introduit des intégrales itérées d'une famille de formes modulaires. Il montre que ces intégrales, pouvant être vu comme des périodes multiples, sont liées entre elles par des relations de mélange ainsi que des relations modulaires.
Le polynôme des périodes permet dans le cas classique, voir Zagier \cite{Za91}, de structurer l'ensemble de ces valeurs pour une forme.
Les relations de Manin permettent alors de décrire l'espace contenant ce polynôme, voir Haberland \cite{Hab83}.
En considérant une généralisation du polynôme des périodes pour deux formes modulaires, on va étudier les relations qu'ils vérifient.
Les résultats de cet article sont ainsi un prolongement des travaux de Manin \cite{Ma2}, Eichler \cite{Eich57} et Shimura \cite{Sh59} pour le polynôme des bi-périodes associé à un couple de formes modulaires.\par

Notons $\Gamma=P\G$ le groupe modulaire. Il agit sur le demi-plan de Poincaré $\H$.

Soit $k\geq 4$ un entier pair. Considérons $S_k$ l'ensemble des formes holomorphes modulaires paraboliques pour $\Gamma$ de poids $k$. 

Notons $w=k-2$ et $V_w^{\Q}=\Q_{w}[X]$ l'anneau des polynômes à coefficients rationnels et de degré au plus $w$.
De plus, notons simplement $V_w=V_w^{\Q}\otimes\C$ le $\C$-espace vectoriel muni de la $\R$-structure donné par $V_w^{\Q}\otimes\R$. Il possède donc une conjugaison complexe.
Tout élément $\gamma=\pm\matg\in\Gamma$ agit à gauche sur $V_w^{\Q}$ par : 
\begin{equation}
\gamma. P(X)=P|_{\gamma^{-1}}(X)=(-cX+a)^{w}P\left(\frac{\phantom{-}dX-b}{-cX+a}\right).
\end{equation}

Pour $f\in S_k$, considérons le polynôme des périodes $P_f\in V_w$ défini par:
\begin{equation}
P_f(X)=\int_0^{\infty}f(it)(X-it)^{w}i\d t=\sum_{m=0}^{w}\binom{w}{m}\Lambda(f,m+1)\frac{X^{w-m}}{i^m}.
\end{equation}
Les \textit{périodes} sont les valeurs aux entiers critiques $1\leq n\leq k-1$ du prolongement holomorphe de la fonction $\Lambda(f,s)=\int_0^{\infty} f(it) t^{s-1}\d t$ définie pour $\Re(s)>k/2+1$.

Les relations de Manin définissent l'idéal de $\Z[\Gamma]$:
\begin{equation}
\mathcal{I}_1=\left\langle 1+S,1+U+U^2\right\rangle,\quad
\text{avec }S=\pm\left(\begin{smallmatrix} 0 & -1\\ 1 & \phantom{-}0\end{smallmatrix}\right)
\text{ et }U=\pm\left(\begin{smallmatrix} \phantom{-}0 & 1\\-1 & 1\end{smallmatrix}\right).
\end{equation}

Notons $W_w^{\Q}=V_w^{\Q}[\mathcal{I}_1]$, l'espace des polynômes annulés par cet idéal. Il contient un élément distingué $P_{G_{w+2}}^+=1-X^{w}$ fourni par la partie paire du polynôme des périodes de la série d'Eisenstein $G_{w+2}$ de poids $w+2$. Notons $\Per_w=\{P_f\text{ pour }f\in S_{w+2}\}$ l'ensemble des polynômes des périodes et $W_w$ l'extension de $W_w^{\Q}$ au corps des complexes.

Le théorème d'Eichler-Shimura \cite{Eich57,Sh59} donne la décomposition en somme directe de $\C$-espaces vectoriels par:
\begin{equation}\label{Wsom}
W_{w}=\Per_w\oplus \overline{\Per_w}\oplus E_w.
\end{equation}

Ainsi $W_w^{\Q}$ est le plus petit $\Q$-espace vectoriel contenant $E_{w}^{\Q}=\Q (1-X^{w})$ et dont l'extension au corps des complexes contienne $\Per_w$.

Soient $w_1,w_2\geq 2$ des entiers pairs.

Notons $V_{w_1,w_2}^{\Q}=\Q_{w_1,w_2}[X_1,X_2]$ l'espace des polynômes en deux indéterminées $X_1$ et $X_2$, à coefficients rationnels et de degrés en $X_1$ et $X_2$ respectivement bornés par $w_1$ et $w_2$. On l'identifiera librement à $V_{w_1}^{\Q}\otimes V_{w_2}^{\Q}$.

De plus, notons simplement $V_{w_1,w_2}=V_{w_1,w_2}^{\Q}\otimes\C$. Ce dernier hérite d'une $\R$-structure donnée par $V_{w_1,w_2}^{\Q}\otimes\R$. 

Le groupe $\Gamma^2$ opère diagonalement à gauche sur $V_{w_1,w_2}^{\Q}$. 
Pour $\gamma_1=\pm\mat{a_1}{b_1}{c_1}{d_1}, \gamma_2=\pm\mat{a_2}{b_2}{c_2}{d_2} \in\Gamma$ et $P\in V_{w_1,w_2}^{\Q}$, on note:
\begin{multline}
(\gamma_1,\gamma_2). P(X_1,X_2)=P(X_1|_{\gamma_1^{-1}},X_2|_{\gamma_2^{-1}})\\
=(-c_1X_1+a_1)^{w_1}(-c_2X_2+a_2)^{w_2}P\left(\frac{\phantom{-}d_1X_1-b_1}{-c_1X_1+a_1},\frac{\phantom{-}d_2X_2-b_2}{-c_2X_2+a_2}\right).
\end{multline}

Soit $(f_1,f_2)\in S_{w_1+2}\times S_{w_2+2}$. Le \textit{polynôme des bi-périodes} $P_{f_1,f_2}\in V_{w_1,w_2}$ du couple des formes $f_1$ et $f_2$ est défini par:
\begin{align}
P_{f_1,f_2}&(X_1,X_2)=\int_{0<t_1<t_2} f_1(it_1)f_2(it_2)(X_1-it_1)^{w_1}(X_2-it_2)^{w_2}i\d t_1i\d t_2\\
&=\sum_{m_1=0}^{w_1}\sum_{m_2=0}^{w_2}\binom{w_1}{m_1}\binom{w_2}{m_2}\Lambda(f_1,f_2;m_1+1,m_2+1)\frac{X_1^{w_1-m_1}X_2^{w_2-m_2}}{i^{m_1+m_2}}.
\end{align}
Les \textit{bi-périodes} sont les valeurs aux couples d'entiers $(n_1,n_2)$ tels que $1\leq n_j\leq k_j-1$ du prolongement analytique à $\C^2$ de l'application définie pour $\Re(s_j)>k_j$ par:
\begin{equation}
\Lambda(f_1,f_2;s_1,s_2)=\int_{0<t_1<t_2} f_1(it_1)f_2(it_2)t_1^{s_1-1}t_2^{s_2-1}\d t_1\d t_2.
\end{equation}

Introduisons $\Per_{w_1,w_2}$ le sous-$\C$-espace vectoriel de $V_{w_1,w_2}$ engendré par les polynômes des bi-périodes, défini par:
\begin{equation}
\Per_{w_1,w_2}=\big\langle P_{f_1,f_2}\text{ tel que }(f_1,f_2)\in S_{w_1+2}\times S_{w_2+2}\big\rangle.
\end{equation}

Définissons un idéal à gauche $\mathcal{I}_2$ de $\Z[\Gamma^2]$ engendré par :
\begin{align}
&(1+S,1+S),
(S,S)+(S,US)+(US,US)+(1,U)-(U^2,U^2),\label{defI2}\\
&(1+U+U^2,1)[(1,1)+(S,S)]\text{ et }
(1,1+U+U^2)[(1,1)+(S,S)].\nonumber
\end{align}
Nous nommerons ces éléments les \textit{relations de Manin d'ordre }$2$.

Considérons le sous-module de $V_{w_1,w_2}^{\Q}$ des polynômes annulés par cet idéal:
\begin{equation}
W_{w_1,w_2}^{\Q}=\big\{P\in V_{w_1,w_2}^{\Q}\text{ tel que }g.P=0,\, \forall g\in\mathcal{I}_2\big\}.
\end{equation}

\begin{prop}\label{prop1}
On a : $\Per_{w_1,w_2}\subset W_{w_1,w_2}=W_{w_1,w_2}^{\Q}\otimes\C$.
\end{prop}

Définissons le sous-espace de $V_{w_1,w_2}^{\Q}$ des polynômes annulés par les relations de Manin diagonales par:
\begin{equation}\label{defVID}
V_{w_1,w_2}^{\Q}[I_D]=\big\{P\in V_{w_1,w_2}^{\Q}\text{ tel que }[(1,1)+(S,S)].P=[(1,1)+(U,U)+(U^2,U^2)].P=0\big\}.
\end{equation}

Par analogie avec le cas d'une forme nécessitant les formes non paraboliques, définissons la partie Eisenstein d'ordre $2$ de $W_{w_1,w_2}^{\Q}$ par:
\begin{equation}\label{defEkk}
E_{w_1,w_2}^{\Q}=\left(W_{w_1}^{\Q}\otimes 1\right)+ V_{w_1,w_2}^{\Q}[I_D]+ \left(X_1^{w_1}\otimes W_{w_2}^{\Q}\right).
\end{equation}

\begin{thm}\label{thm1}
L'espace vectoriel $W_{w_1,w_2}^{\Q}$ est le plus petit sous-$\Q$-espace-vectoriel de $V_{w_1,w_2}^{\Q}$ contenant $E_{w_1,w_2}^{\Q}$ tel que son extension au corps des complexes contient $\Per_{w_1,w_2}$.
\end{thm}

La suite exacte de $\Q$-espaces vectoriels suivantes précise l'écart qu'il existe entre les espaces $E_{w_1,w_2}^{\Q}$ et $W_{w_1,w_2}^{\Q}$:
\begin{equation}\label{suiteex}
0\to E_{w_1,w_2}^{\Q} \to W_{w_1,w_2}^{\Q} \stackrel{\Phi_{S}}{\longrightarrow} W_{w_1}^{\Q}/E_{w_1}^{\Q} \otimes W_{w_2}^{\Q}/E_{w_2}^{\Q},
\end{equation}
avec $\Phi_{S} : P\mapsto\left[[(1,1)+(S,S)].P\right]$.

L'application $\Phi_S$ n'est pas toujours surjective. Dans les cas de dimension raisonnable ($w\leq 24$), le calcul numérique montre que $\Phi_S$ est surjective sauf lorsque $w_1=w_2$ et $S_{w_1+2}\neq\{0\}$. La généralisation pour des poids quelconques n'est pas démontré dans cet article. Toutefois, la surjectivité observée numériquement couplé au théorème \ref{thm1} permettent d'obtenir un résultat sur l'irrationalité des périodes.

Pour $w\in\{10, 14, 16, 18, 20, 24\}$, l'espace $S_{w+2}$ est de dimension $1$ et est engendré par une forme modulaire parabolique de Hecke $f_w$. La théorie de Eichler-Shimura \cite{Eich57,Sh59} permet d'écrire : 
\begin{equation}
P_{f_w}(X)=\Omega_w^{+}P_w^+(X)+i\Omega_w^{-}P_w^-(X),
\end{equation} 
avec $P_w^+(X),P_w^-(X)\in V_w^{\Q}$ respectivement pair et impair. Les réels $\Omega_w^+$ et $\Omega_w^-$ sont appelés les périodes de la forme $f_w$. Leur rapport $r_w=\frac{\Omega_w^+}{\Omega_w^-}$ ne dépend alors que du poids $w$.

\begin{thm}\label{thm2}
Soient $w_1,w_2$ des poids tels que $S_{w_1+2}$ et $S_{w_2+2}$ sont de dimension $1$. 
Alors le produit $r_{w_1}r_{w_2}$ est irrationnel.
De plus si $w_1\neq w_2$ alors le rapport $r_{w_1}/r_{w_2}$ est irrationnel.
\end{thm}

En particulier, ceci démontre que pour un poids $w$ tel que $S_{w+2}$ est de dimension $1$ le rapport $r_w$ est irrationnel.\par

Pour démontrer ces résultats, nous allons étudier l'idéal à gauche de $\Z[\Gamma^2]$:
\begin{equation}
\mathcal{J}(k_1,k_2)=\{g\in\Z[\Gamma^2]\text{ tel que }g.P_{f_1,f_2}=0,\text{ pour tout }(f_1,f_2)\in S_{k_1}\times S_{k_2}\}.
\end{equation}

Dans une première partie, on va donner une description, indépendante des poids $k_1$ et $k_2$, de l'idéal à gauche $\mathcal{J}(k_1,k_2)$ défini entièrement par les propriétés topologiques de $\H^2$ sous l'action de $\Gamma^2$.
Ceci donnera les arguments principaux de la démonstration du Théorème \ref{thm1}.\par

Dans une seconde partie, on va relier le calcul de $\mathcal{I}_2$ à celui de $\mathcal{I}_1$.
Puis nous montrerons que $\mathcal{I}_2$ est de type fini. 
Nous démontrons alors que les relations (\ref{defI2}) en constituent une famille de générateurs.\par

Dans la dernière partie, nous démontrons la validité de la suite exacte (\ref{suiteex}). 
Nous donnons ensuite une description calculatoire de la partie d'Eisenstein d'ordre $2$ et ainsi un calcul de sa dimension. 
Les valeurs obtenues dans les cas $w\in\{10,14,16,18,20,24\}$, nous permettrons alors de démontrer le théorème \ref{thm2}.\par


\section{Définition homologique de l'idéal des relations doubles}

\subsection{Homologie singulière relative aux pointes}

Nous introduisons une homologie singulière. Pour référence, on pourra se rapporter au livre de Hatcher \cite{Hat}.
Soit $X$ un espace topologique et $X_0$ une partie fermée de $X$. Soit $m\geq 0$ un entier. On définit le simplexe fondamental de dimension $m$ par:
\begin{equation*}
\Delta_m=\{(t_0,...,t_m)\in [0,1]^{m+1}\text{ tel que }\sum_{j=0}^m t_j =1\},
\end{equation*}
et l'ensemble de ses sommets par: $\Delta_m^0=\Delta_m\cap \Z^{m+1}=\{e_j^m=(...,0,1,0,...),j=0...n\}$.\par
Définissons $M_m(X,X_0,\Z)$ le $\Z$-module libre engendré par les $m$-cycles de $X$ aux sommets dans $X_0$ à homotopie près. C'est-à-dire la classe des applications continues:
\begin{equation*}
C:\Delta_m\to X,\text{ telle que }C(\Delta_m^0)\subset X_0.
\end{equation*}
On dit que $C_0$ et $C_1$ sont homotopes s'il existe une application continue:
\begin{equation*}
h:\Delta_m\times [0,1]\to X\text{ tel que }h(u,0)=C_0(u)\text{ et }h(u,1)=C_1(u)\text{ pour tout }u\in\Delta_m,
\end{equation*}
et pour tout $t\in [0,1]$ et $p_0\in\Delta_m^0$, on a $h(p_0,t)\in X_0$.\par
Ceci permet de considérer les applications de bord:
\begin{equation}
\delta_m: M_{m}(X,X_0,\Z)\to M_{m-1}(X,X_0,\Z),\quad [C] \mapsto \sum_{j=0}^{m} (-1)^j [C\circ\delta_j^m],
\end{equation}
où $\delta_j^m:\Delta_{m-1}\to\Delta_{m}, (t_0,...,t_{m-1})\mapsto (t_0,...,t_{j-1},0,t_{j},...,t_{m-1})$.
Elles vérifient pour tout entier $m\geq 0$, $\delta_m\circ\delta_{m+1}=0$. Ceci permet de considérer les groupes d'homologie singulière:
\begin{equation}
H_{m}(X,X_0,\Z)=\Ker\left(\delta_m|M_{m}(X,X_0,\Z)\right)/\Im\left(\delta_{m+1}|M_{m}(X,X_0,\Z)\right).
\end{equation}

Dans notre cadre, nous étudierons des parties $X$ de $\H$ ou $\H^2$. Elles seront associées à des ensembles de sommets $X_0$ retreints à leurs pointes respectif $X_0=X\cap\pte$ ou $X_0=X\cap\pte^2$. Pour toute partie $X$ relative à $X_0$ de ce type nous noterons simplement:
\begin{equation}
M_m^{pte}(X,\Z)=M_m(X,X_0,\Z)\text{ et }H_m^{pte}(X,\Z)=H_m(X,X_0,\Z).
\end{equation}

Nous remarquons que ces espaces de sommets $X_0$ sont discrets. La continuité des restrictions $h_0:\Delta_m^0\times [0,1]\to X_0$ démontre alors qu'elles sont constantes suivants la seconde variable. Ainsi dans le cadre de notre étude, les images des sommets ne dépendent pas du représentant mais seulement de la classe d'homotopie.

De plus, on peut préciser le bord formel de $\H$. Il est donné par $\partial\H=\pte$ de sorte qu'il apparaisse comme la limite des $\gamma z$ lorsque $\gamma\in\Gamma$ et $z\to 0$.
On adopte alors les notations $(p:q)=\frac{p}{q}$ lorsque $q\neq 0$ et $(1:0)=\infty$.\par

L'exemple principal utile dans cet article est le groupe $M_2^{pte}(\H^2,\Z)$.  Il est le $\Z$-module librement engendré par les $2$-cycles de $\H^2$ aux sommets dans $\pte^2$. Un tel $2$-cycle est une classe d'équivalence, aux homotopies fixant les sommets près, des applications continues :
$$C:\Delta_2=\{(t_0,t_1,t_2)\in[0,1]^3\text{ tel que }t_0+t_1+t_2=1\}\to \H^2,$$
vérifiant $C(1,0,0),C(0,1,0),C(0,0,1)\in\pte^2$. La classe de l'application:
\begin{equation}
\Delta_2\to\H^2,\quad(t_0,t_1,t_2)\mapsto \left(-i\log(t_1+t_2),-i\log(t_2)\right),
\end{equation}
définit un élément $\tau_2$ de $M_2^{pte}(\H^2,\Z)$ car on a dans $\pte^2$ :
$$\tau_2(1,0,0)=(\infty,\infty),\quad \tau_2(0,1,0)=(0,\infty)\text{ et }\tau_2(0,0,1)=(0,0).$$

Le polynôme des bi-périodes du couple $(f_1,f_2)$ peut être écrit en fonction du $2$-cycle $\tau_2$ par:
\begin{equation}
P_{f_1,f_2}(X_1,X_2)=\int_{\tau_2} f_1(z_1)(X_1-z_1)^{w_1}f_2(z_2)(X_2-z_2)^{w_2}\d z_1\d z_2. 
\end{equation}

On remarque alors l'intérêt d'étudier les $2$-formes différentielles $\omega_{f_1,f_2}:\H^2\to V_{w_1,w_2}$ définies par:
\begin{equation}
\omega_{f_1,f_2}(z_1,z_2) = f_1(z_1)(X_1-z_1)^{w_1}f_2(z_2)(X_2-z_2)^{w_2}\d z_1\d z_2.
\end{equation}

\subsection{Les formes différentielles associées aux formes modulaires}

Soit $\Omega_{par}^2(\H^2,\C)$ le $\C$-espace vectoriel des $2$-formes différentielles harmoniques de $\H^2$ et nulles aux pointes $\pte^2$. 
Cette propriété permet notamment d'obtenir la convergence de l'intégration d'une telle forme le long de $\tau_2$.\par
Soient $(\gamma_1,\gamma_2)\in\Gamma^2$. Introduisons $\psi_{\gamma_1,\gamma_2}:\H^2\to\H^2$ définie pour $(z_1,z_2)\in\H^2$ par :
\begin{equation}
\psi_{\gamma_1,\gamma_2}(z_1,z_2)=(\gamma_1 z_1,\gamma_2 z_2).
\end{equation}
Considérons une action à gauche de $\Gamma^2$ sur $\Omega_{par}^2(\H^2,\C)$, en posant:
\begin{equation}
(\gamma_1,\gamma_2).\omega (z_1,z_2) = \left(\psi_{\gamma_1^{-1},\gamma_2^{-1}}^* \omega\right)(z_1,z_2) = \omega(\gamma_1^{-1}z_1,\gamma_2^{-1}z_2).
\end{equation}

D'autre part, l'action diagonale de $\Gamma^2$ sur $\H^2$ fournit une action sur $M_2^{pte}(\H^2,\Z)$.\par

Considérons la forme bilinéaire définie par:
\begin{equation}\label{defacc}
\Omega_{par}^2(\H^2,\C)\times M_2^{pte}(\H^2,\C)\to\C,\quad
(\omega,C)\mapsto \langle\omega, C\rangle=\int_{C}\omega.
\end{equation}

Cette forme est non dégénérée et invariante sous l'action de $\Gamma^2$. Plus précisément, pour toute $2$-forme différentielle $\omega\in\Omega^2_{par}(\H^2,\C)$, toute $2$-chaîne $C\in M_2^{pte}(\H^2,\Z)$ et $\gamma\in\Gamma^2$, on a :
\begin{equation}
\langle\gamma.\omega ,\gamma.C \rangle=\int_{\gamma.C}\omega(\gamma^{-1}z)=\int_{C}\omega(z)=\langle\omega,C\rangle.\label{relacc31}
\end{equation}

Pour tous les couples $(f_1,f_2)\in S_{w_1+2}\times S_{w_2+2}$ et $(m_1,m_2)\in\Z^2$ vérifiant $0\leq m_j\leq w_j$), les $2$-formes :
$$f_1(z_1)z_1^{m_1}f_2(z_2)z_2^{m_2}\d z_1\wedge \d z_2\in\Omega_{par}^2(\H^2,\C)$$
peuvent être indexée par $X_1^{w_1-m_1}X_2^{w_2-m_2}$. Après renormalisation, elles définissent bien :
\begin{equation*}
\omega_{f_1,f_2}(z_1,z_2)=\sum_{m_1=0}^{w_1}\sum_{m_2=0}^{w_2}\binom{w_1}{m1}\binom{w_2}{m_2} f(z_1)z_1^{m_1}f_2(z_2)z_2^{m_2}\d z_1\wedge \d z_2 (-X_1)^{w_1-m_1}(-X_2)^{w_2-m_2}.
\end{equation*}
Ainsi les formes $\omega_{f_1,f_2}$ sont des éléments de  $\Omega_{par}^2(\H^2,\C)\otimes_{\Q} V_{w_1,w_2}=\Omega_{par}^2(\H^2,V_{w_1,w_2})$.\par

On remarquera que cette construction est valide dans le cas classique où l'on défini pour $f\in S_{w+2}$ la $1$-forme : 
\begin{equation}
\omega_f(z)=f(z)(X-z)^{w}\d z \in\Omega_{par}^1(\H,V_{w}).
\end{equation}
En particulier, on a : $\omega_{f_1,f_2}(z_1,z_2)=\omega_{f_1}(z_1)\wedge\omega_{f_2}(z_2)$.\par

On peut observer une propriété d'invariance sous l'action de $\Gamma$. Pour tout $\gamma=\pm\matg\in\Gamma$ et toute forme $f\in S_{w+2} $, la forme différentielle associée $\omega_f$ vérifie:
\begin{equation}
\gamma.\omega_f(\gamma^{-1}z,X) = (-cX+a)^w f(\gamma^{-1}z)\left(\gamma^{-1}z-\gamma^{-1}X\right)^w\d (\gamma^{-1}z) =\omega_f(z,X).\label{relmod31}
\end{equation}

Ceci permet de déduire la $\Gamma^2$-invariance des formes $\omega_{f_1,f_2}$. Et ainsi de transporter l'action sur $V_{w_1,w_2}$ en une action sur $M_2^{pte}(\H^2,\Z)$ dans le cas des polynômes des bi-périodes.

\begin{prop}\label{prop37}
Les actions de $\Gamma^2$ sur $V_{w_1,w_2}$ et $ M_2^{pte}(\H^2,\Z)$ sont liées par:
\begin{equation}
g.P_{f_1,f_2}=\langle\omega_{f_1,f_2}, g.\tau_2\rangle,\text{ pour tout } g\in\Z[\Gamma^2].
\end{equation}
\end{prop}

\begin{proof}
Soit $(\gamma_1,\gamma_2)\in\Gamma^2$. Son action sur le polynôme des bi-périodes devient:
\begin{align*}
{(\gamma_1,\gamma_2)}.P_{f_1,f_2}(X_1,X_2)
&=\langle \gamma_1.\omega_{f_1}(z_1,X_1)\wedge\gamma_2.\omega_{f_2}(z_2,X_2),\tau_2\rangle\\
&=\langle \omega_{f_1}(\gamma_1.z_1,X_1)\wedge\omega_{f_2}(\gamma_2.z_2,X_2),\tau_2\rangle\text{ par (\ref{relmod31})}\\
&=\langle\omega_{f_1}(z_1,X_1)\wedge\omega_{f_2}(z_2,X_2),(\gamma_1,\gamma_2).\tau_2\rangle\text{ par (\ref{relacc31})}.
\end{align*}
La proposition s'étend par linéarité à $\Z[\Gamma^2]$.
\end{proof}

Définissons les sous-espaces suivants de $\Omega_{par}^2(\H^2,V_{w_1,w_2})$ par:
\begin{align}
\Omega_{k_1,k_2}^{+}&=\left\langle \omega_{f_1,f_2} \text{ pour }(f_1,f_2)\in S_{k_1}\times S_{k_2}\right\rangle,\\
\Omega_{k_1,k_2}^{-}&=\left\langle \omega_{f_1,f_2} \text{ pour }(\overline{f_1},\overline{f_2})\in S_{k_1}\times S_{k_2}\right\rangle.
\end{align}
Le premier est constitué de $2$-formes holomorphes en les deux variables et le second de $2$-formes anti-holomorphes en les deux variables. De plus, on peut remarquer qu'ils sont conjugués complexes l'un de l'autre. Considérons alors leur somme directe stable par conjugaison complexe:
\begin{equation}
\Omega_{k_1,k_2}^{\harm}=\Omega_{k_1,k_2}^{+}\oplus\Omega_{k_1,k_2}^{-}.\label{sumomega2}
\end{equation}

Définissons pour tout ensemble $\Omega$ de $2$-formes harmoniques sur $\H^2$, son orthogonal dans $M_2^{pte}(\H^2,\Z)$ par:
\begin{equation}
\Omega^{\bot}=\big\{C\in M_2^{pte}(\H^2,\Z)\text{ tel que }\langle \omega,C\rangle=0,\text{ pour tout }\omega\in\Omega\big\}.
\end{equation}

Ceci permet de réécrire l'idéal $\mathcal{J}(k_1,k_2)$:

\begin{coro}\label{propJkk}
On a $\mathcal{J}(k_1,k_2)=\big\{g\in \Z[\Gamma^2]\text{ tel que }g.\tau_2\in \left(\Omega_{k_1,k_2}^{\harm}\right)^{\bot}\big\}$.
\end{coro}

\begin{proof}
La Proposition \ref{prop37} donne pour tout élément $g\in\Z[\Gamma^2]$:
\begin{align*}
\forall (f_1,f_2)\in S_{k_1} \times S_{k_2} , g.P_{f_1,f_2}=0
&\Leftrightarrow \forall (f_1,f_2)\in S_{k_1} \times S_{k_2} ,\langle \omega_{f_1,f_2}, g.\tau_2\rangle=0\\
\Leftrightarrow \forall \omega\in \Omega_{k_1,k_2}^{+}, \langle \omega, g.\tau_2\rangle=0
&\Leftrightarrow g.\tau_2\in \left(\Omega_{k_1,k_2}^{+}\right)^{\bot}.
\end{align*}
On remarque que : $\left(\Omega_{k_1,k_2}^{+}\right)^{\bot}=\left(\overline{\Omega_{k_1,k_2}^{+}}\right)^{\bot}=\left(\Omega_{k_1,k_2}^{-}\right)^{\bot}$ d'après le calcul:
$$\forall\omega\in\Omega^2,\forall C\in M_2,\langle \omega,C\rangle=0 \Leftrightarrow \langle \overline{\omega},C\rangle=\overline{\langle \omega, C\rangle}=0.$$
Ce permet d'obtenir: 
$$\left(\Omega_{k_1,k_2}^{+}\right)^{\bot}
=\left(\Omega_{k_1,k_2}^{+}\right)^{\bot}\cap\left(\Omega_{k_1,k_2}^{-}\right)^{\bot}
=\left(\Omega_{k_1,k_2}^{\harm}\right)^{\bot}.$$
La dernière égalité provenant de la définition (\ref{sumomega2}).
\end{proof}

\subsection{Espaces transverses de $ M_1^{pte}(\H^2,\Z)$}

Pour calculer $(\Omega_{k_1,k_2}^{\harm})^{\bot}$, nous introduisons des \textit{espaces transverses} de $M_1^{pte}(\H^2,\Z)$.

Soient $g\in\Gamma$ et $c\in M_1^{pte}(\H,\Z)$. 
Définissons $H_g$, $V_g$ et $D_g$ des applications de $M_1^{pte}(\H,\Z)$ dans $M_1^{pte}(\H^2,\Z)$ données respectivement par les classes des applications:
\begin{align}
H_g(c):\Delta_1\to\H^2,\quad& (t_0,t_1) \mapsto (c(t_0,t_1),g.\infty),\\
V_g(c):\Delta_1\to\H^2,\quad& (t_0,t_1) \mapsto (g.0,c(t_0,t_1)),\\
\text{et }D_g(c):\Delta_1\to\H^2,\quad& (t_0,t_1) \mapsto (c(t_0,t_1),g.c(t_0,t_1)).
\end{align}

Considérons les sous-modules de $M_1^{pte}(\H^2,\Z)$ définis par :
\begin{align}
H=&\{H_g(c)\text{ pour }c\in M_1^{pte}(\H,\Z)\text{ et }g\in\Gamma\},\\
V=&\{V_g(c)\text{ pour }c\in M_1^{pte}(\H,\Z)\text{ et }g\in\Gamma\},\\
\text{et }D=&\{D_g(c)\text{ pour }c\in M_1^{pte}(\H,\Z)\text{ et }g\in\Gamma\}.
\end{align}

On qualifiera de \textit{transverses} les chaînes de $H+V+D\subset M_1^{pte}(\H^2,\Z)$.\par

Notons $H^0$, $V^0$ et $D^0$ les sous-groupes constitués par les chaînes fermées, c'est-à-dire de bord nul, de $H$, $V$ et $D$ respectivement.\par

L'action de $\Gamma^2$ sur $ M_1^{pte}(\H^2,\Z)$ respecte les ensembles $H$, $D$ et $V$ :
\begin{align}
(\gamma_1,\gamma_2).H_g(c)&=H_{\gamma_2 g}(\gamma_1.c),\label{acth}\\
(\gamma_1,\gamma_2).V_g(c)&=V_{\gamma_1 g}(\gamma_2.c),\\
\text{et }(\gamma_1,\gamma_2).D_g(c)&=D_{\gamma_2 g \gamma_1^{-1}}(\gamma_1.c).
\end{align}

Notons $\Gamma_{\infty}$ et $\Gamma_0$ les sous-groupes de $\Gamma$ stabilisant les pointes $\infty$ et $0$ respectivement.
L'action de $\Gamma$ scinde chacun de ces espaces transverses:
\begin{align}
H&=\bigoplus_{g\in\Gamma/\Gamma_{\infty}} H_g(M_1^{pte}(\H,\Z)),\label{sumh}\\
V&=\bigoplus_{g\in\Gamma/\Gamma_{0}} V_g( M_1^{pte}(\H,\Z)),\\
\text{et }D&=\bigoplus_{g\in\Gamma} D_g( M_1^{pte}(\H,\Z)).\label{sumd}
\end{align}

Les sous-groupes $H,D$ et $V$ sont deux à deux d'intersection nulle.
En effet on résout les différentes équations suivantes:
\begin{align*}
(g_1.\infty,c_1(t_0,t_1))&=(c_2(t_0,t_1),g_2.0), \forall (t_0,t_1)\in\Delta_1 &\text{pour }V\cap H,\\
(g_1.\infty,c_1(t_0,t_1))&=(c_2(t_0,t_1),g_2.c_2(t_0,t_1)), \forall (t_0,t_1)\in\Delta_1 &\text{pour }V\cap D,\\
\text{et }(c_1(t_0,t_1),g_1.c_1(t_0,t_1))&=(c_2(t_0,t_1),g_2.0), \forall (t_0,t_1)\in\Delta_1 &\text{pour }D\cap H.
\end{align*}
Chacune de ces équations fixe les chemins $c_1$ et $c_2$ comme étant constant. Donc la classe de $(c_1,c_2)$ est nulle dans $M_1^{pte}(\H^2,\Z)$.

Pour $0\leq j\leq 2$, posons $\varphi_j:\H\to\H^2$ définies pour $z\in\H$ par:
\begin{equation}\label{applitrans}
\varphi_0(z)=(z,\infty),\quad\varphi_1(z)=(z,z)\text{ et}\quad\varphi_2(z)=(0,z).
\end{equation}

Ces fonctions se prolongent bien aux pointes et on obtient les écritures:

\begin{equation*}
\psi_{\gamma_1,\gamma_2}(\varphi_0(c)) = H_{\gamma_1}(\gamma_2 c),\quad
\psi_{\gamma_1,\gamma_2}(\varphi_1(c)) = D_{\gamma_1}(\gamma_2 c)\text{ et }
\psi_{\gamma_1,\gamma_2}(\varphi_2(c)) = V_{\gamma_1}(\gamma_2 c).
\end{equation*}

On considère $\tau_1\in M_1^{pte}(\H,\Z)$ la classe de l'application définie par :
\begin{equation}
\tau_1(t_0,t_1)=-i\log (t_1)\text{ pour }(t_0,t_1)\in\Delta_1.
\end{equation}
On retrouve le polynôme des périodes d'une forme $f\in\S_k$ par : $P_f=\int_{\tau_1}\omega_f$.

\begin{rem}\label{dtau2}
La $1$-chaîne $\delta_2 \tau_2$ est la somme alternée de trois $1$-cycles de $M_1^{pte}(\H^2,\Z)$:
\begin{equation}\label{dtau2eq}
\delta_2 \tau_2=H_{id}(\tau_1)-D_{id}(\tau_1)+V_{id}(\tau_1)=\sum_{j=0}^2(-1)^j\varphi_j(\tau_1).
\end{equation}
\end{rem}

En effet les bords d'une $2$-chaine sont donnés par la formule $\sum_{j=0}^2(-1)^j\delta_2^j$ où $\delta_2^j$ consiste à imposer $t_j=0$, ainsi le bord de $\tau_2$ est déterminé par:
\begin{align*}
\tau_2(t_0,t_1,0)&=\left(-i\log(t_1),\infty\right)=\left(\tau_1(t_0,t_1),\infty\right),\\
\tau_2(t_0,0,t_2)&=\left(-i\log(t_2),-i\log(t_2)\right)=\left(\tau_1(t_0,t_2),\tau_1(t_0,t_2)\right),\\
\text{et }\tau_2(0,t_1,t_2)&=\left(0,-i\log(t_2)\right)=\left(0,\tau_1(t_1,t_2)\right).
\end{align*}
Ce qui donne la décomposition de $\delta_2\tau_2$ en fonction de $\tau_1$ après réécriture.\par

Dans un premier temps, nous introduisons l'idéal $\mathcal{J}_2$ par :
\begin{equation}\label{defI2cycle}
\mathcal{J}_2=\big\{g\in\Z[\Gamma^2]\text{ tel que }g.\delta_2\tau_2\in H^0+V^0+D^0\big\}.
\end{equation}
Cette définition est clairement indépendante des poids $(k_1,k_2)$. 
Dans la seconde section, nous démontrerons qu'il est de type fini et qu'il est bien engendré par les relations (\ref{defI2}) de Manin d'ordre $2$. 
On aura alors $\mathcal{J}_2=\mathcal{I}_2$.\par

On définit les idéaux annulateurs des chemins transverses par:
\begin{align}
I_H&=\{g\in\Z[\Gamma^2]\text{ tel que }g.H_{id}(\tau_1)\in H^0\},\label{defIH}\\ 
I_V&=\{g\in\Z[\Gamma^2]\text{ tel que }g.V_{id}(\tau_1)\in V^0\},\label{defIV}\\
\text{et }I_D&=\{g\in\Z[\Gamma^2]\text{ tel que }g.D_{id}(\tau_1)\in D^0\}.\label{defID}
\end{align}

\begin{prop}\label{inclhvd}
L'idéal $\mathcal{J}_2$ de $\Z[\Gamma^2]$ est déterminé par l'intersection des idéaux annulateurs de ces trois segments:
\begin{equation}
\mathcal{J}_2=I_H\cap I_V \cap I_D.
\end{equation}
\end{prop}

\begin{proof}
L'inclusion $I_H\cap I_V \cap I_D\subset \mathcal{J}_2$ est une conséquence directe de la décomposition de $\delta_2 \tau_2$ vu en (\ref{dtau2eq}). Réciproquement, soit $g\in\mathcal{J}_2$ alors $g.\delta_2 \tau_2\in (H^0+V^0+D^0)$. On a:
$$g.H_{id}(\tau_1)=g.\delta_2\tau_2-g.V_{id}(\tau_1)+g.D_{id}(\tau_1)\in H\cap\left(H^0+V+D\right).$$
car les chaînes transverses sont stables par $\Gamma$. Puis on a :
$$H\cap\left(H^0+V+D\right)=H^0+(H\cap V)+(H\cap D)=H^0.$$
En effet, les $\Z$-modules $H$, $V$ et $D$ sont deux à deux d'intersection nulle. De plus, ils sont sans torsions donc quitte à prendre les $\Q$-espaces vectoriels associés, on obtient bien l'égalité. Et ainsi on obtient $g\in I_H$. Ce raisonnement se symétrise et on obtient bien le résultat.
\end{proof}

Considérons alors le $\Z[\Gamma^2]$-module des $2$-cycles à bords transverses défini par:
\begin{equation}\label{defM2tr}
M_2^{\tr}(\H^2,\Z)=\delta_2^{-1}( H+V+D ) = \left\{C\in M_2^{pte}(\H^2,\Z) \text{ tel que }\delta_2 C\in H+V+D\right\}.
\end{equation}
Il contient en particulier $\tau_2$ ainsi que le sous-module suivant:
\begin{equation}\label{defM2tr0}
M_2^{0}(\H^2,\Z)=\delta_2^{-1}( H^0+V^0+D^0 ) = \left\{C\in M_2^{pte}(\H^2,\Z) \text{ tel que }\delta_2 C\in H^0+V^0+D^0\right\}.
\end{equation}

On dispose ainsi de l'écriture :
\begin{equation}\label{defI2cycletr}
\mathcal{J}_2=\big\{g\in\Z[\Gamma^2]\text{ tel que }g.\tau_2\in M_2^{0}(\H^2,\Z)\big\}.
\end{equation}

De plus, on peut démontrer que : 
\begin{equation}\label{sumM2tr0}
M_2^{0}(\H^2,\Z)=\Ker(\delta_2)+\sum_{\gamma\in\Gamma^2}\sum_{j=0}^2 \psi_{\gamma}\varphi_j\left(M_2^{pte}(\H,\Z)\right).
\end{equation}

\subsection{Formes différentielles transverses sur $\Gamma^2\backslash\H^2$}

Considérons le sous-espace $\Omega_{k_1,k_2}$ des $2$-formes $\omega\in\Omega^2_{par}(\H^2,V_{w_1,w_2})$ vérifiant les trois propriétés :
\begin{itemize}
\item $\omega$ est une forme fermée : $\d\omega=0$,
\item $\omega$ est invariante par le groupe $\Gamma^2$ : $\forall \gamma\in\Gamma^2, \psi_{\gamma}^*\omega=\gamma.\omega$,
\item $\omega$ est transverse : $\varphi_0^*\omega=\varphi_1^*\omega=\varphi_2^*\omega=0$.
\end{itemize}

De plus, on notera $\Omega_{k_1,k_2}^0$ l'ensemble des formes exactes de $\Omega_{k_1,k_2}$.

\begin{ex}
\textbf{a)}
Les $2$-formes $\omega_{f_1,f_2}$ constituent des exemples de référence. 
En effet, un calcul direct nous montre que les $2$-formes associées aux formes modulaires vérifiées les trois propriétés précédentes. 
De plus, la conjugaison complexe commute avec les trois conditions. 
Ainsi $\Omega_{k_1,k_2}^{\harm}$ est un sous-espaces de $\Omega_{k_1,k_2}$.

\textbf{b)}
On construit désormais des exemples pour les formes exactes.
On dispose de $G_{w+2}(z)$ la forme d'Eisenstein de poids $w+2$. Elle permet de construire une $1$-forme exacte sur $\H$ et à valeurs dans $V_w$ par :
\begin{equation}
\omega_w^0(z)=G_{w+2}(z)(X-z)^w\d z - G_{w+2}(-\bar{z})(X-\bar{z})^w\d \bar{z} \in \Omega_{k}^0.
\end{equation}

L'écriture en somme $G_{w+2}(z)=\sum_{\mat{*}{*}{c}{d}\in\Gamma_{\infty}\backslash\Gamma} (cz+d)^{-w-2}$ permet d'obtenir le caractère exact.
En effet, on a:
\begin{align*}
\omega_w^0(z)&=\sum_{\mat{*}{*}{c}{d}\in\Gamma_{\infty}\backslash\Gamma} \left(\frac{(X-z)^w}{(cz+d)^{w}}\frac{\d z}{(cz+d)^{2}} - \frac{(X-\bar{z})^w}{(c\bar{z}+d)^{w}}\frac{\d \bar{z}}{(c\bar{z}+d)^{2}}\right)\\
&=\sum_{\gamma=\mat{*}{*}{c}{d}\in\Gamma_{\infty}\backslash\Gamma} (cX+d)^w \left((\gamma X-\gamma z)^w\d (\gamma z) - (\gamma X-\gamma \bar{z})^w\d (\gamma \bar{z})\right)\\
&=\d\left(\sum_{\gamma=\mat{*}{*}{c}{d}\in\Gamma_{\infty}\backslash\Gamma} \frac{(cX+d)^w}{w+1}\left( (\gamma z-\gamma X)^{w+1} - (\gamma \bar{z}-\gamma X)^{w+1}\right)\right).
\end{align*}

Nous introduisons ainsi la fonction harmonique dont dérive $\omega_w^0$:
\begin{align}
F_w(z,X) &= \sum_{\gamma=\mat{*}{*}{c}{d}\in\Gamma_{\infty}\backslash\Gamma} \frac{(cX+d)^w}{w+1}\left( (\gamma z-\gamma X)^{w+1} - (\gamma \bar{z}-\gamma X)^{w+1}\right).\\
&=  \frac{z-\bar{z}}{w+1}\sum_{\alpha=0}^w \sum_{\mat{*}{*}{c}{d}\in\Gamma_{\infty}\backslash\Gamma}\frac{(z-X)^{\alpha}(\bar{z}-X)^{w-\alpha}}{(cz+d)^{\alpha+1}(c\bar{z}+d)^{w-\alpha+1}}.
\end{align}

Elle est bien à valeurs dans $V_w$ et est $\Gamma$-invariant car vérifie:
\begin{equation}
\forall g\in\Gamma, \psi_{g}^*F_w(z,X)=F_w(gz,X)=g.F_w(z,X).
\end{equation}

Cette $1$-forme exacte permet de considérer des exemples de $2$-formes exactes de $\Omega_{k_1,k_2}^0$ par:
\begin{align}
\omega_{w_1}^0\wedge\omega_{f_2} = \d\Big(F_{w_1}(z_1,X_1) f_2(z_2)(X_2-z_2)^{w_2}\d z_2\Big)\text{ pour }f_2\in S_{w_2+2},\\
\omega_{f_1}\wedge\omega_{w_2}^0 = \d\Big(f_1(z_1)(X_1-z_1)^{w_1}F_{w_2}(z_2,X_2)\d z_1\Big)\text{ pour }f_1\in S_{w_1+2}.
\end{align}
\end{ex}

On va désormais construire un produit scalaire de Petersson d'ordre $2$ sur $\Omega_{k_1,k_2}$.\par

Haberland (p.277 dans \cite{Hab83}) propose de munir l'espace $V_w^{\Q}$ d'un produit scalaire défini par :
$\left[ X^m, X^{w-m}\right] = (-1)^m \binom{w}{m}^{-1}$ et nul sur les autres paires de monômes.\par

Ce produit scalaire est invariant par l'action de $\Gamma$ :
\begin{equation}\label{invg2}
\left[ \gamma.P,\gamma.Q\right] = \left[P, Q\right],\text{ pour }P,Q\in V_w\text{ et }\gamma\in\Gamma.
\end{equation}

De plus, par construction, il vérifie:
\begin{equation}\label{Pol1}
[(X-a)^w,P] = P(a)\quad\text{ pour }P\in V_w\text{ et }a\in\C.
\end{equation}

Cette dernière remarque permet notamment de retrouver le produit scalaire de Petersson.
Pour $f,g\in S_{w+2}$, on a : $(f,g)=\int_{\Gamma\backslash\H}[\omega_f,\overline{\omega_g}]$ car :
\begin{equation}
[\omega_f,\overline{\omega_g}] = f(z)\overline{g(z)} [(X-z)^w,(X-\bar{z})^w] \d z\d \bar{z}=f(z)\overline{g(z)} (z-\bar{z})^w \d z\d \bar{z}.
\end{equation}\par

On peut alors construire par tensorisation une forme bilinéaire non dégénérée sur $V_{w_1,w_2}^{\Q}$ et invariante par $\Z[\Gamma^2]$ par:
$\left[P_1\otimes P_2, Q_1\otimes Q_2\right] = \left[P_1,Q_1\right].\left[P_2,Q_2\right]$.

Ainsi on peut définir un produit de Petersson d'ordre $2$ pour $\omega_1,\omega_2\in\Omega_{k_1,k_2}$ par:
\begin{equation}
(\omega_1,\omega_2)_{P} = \int_{\Gamma^2\backslash\H^2} [\omega_1,\overline{\omega_2}] \in\C.
\end{equation}

On peut vérifier que pour deux telles $2$-formes, la $4$-forme $[\omega_1,\overline{\omega_2}]$ est bien définie sur $\Gamma^2\backslash\H^2$ car pour $\gamma\in\Gamma^2$ :
\begin{equation}
\psi_{\gamma}^*[\omega_1,\overline{\omega_2}]=[\psi_{\gamma}^*\omega_1,\psi_{\gamma}^*\overline{\omega_2}]=[\gamma.\omega_1,\gamma.\overline{\omega_2}]=[\omega_1,\overline{\omega_2}].
\end{equation}

On remarque en particulier que pour deux couples $(f_1,f_2),(g_1,g_2)\in S_{k_1}\times S_{k_2}$, on retrouve le produit de Petersson classique :
\begin{equation}
(\omega_{f_1,f_2},\omega_{g_1,g_2})_{P} = (f_1,g_1)(f_2,g_2)
\end{equation}

D'autre part, on observe que $\Omega_{k_1,k_2}^0$ est le noyau de cette forme bilinéaire. En effet, on peut introduire un domaine fondamentale $D_2$ pour $\Gamma^2\backslash\H^2$ et pour deux formes $\d F,\omega\in\Omega_{k_1,k_2}$, on a:
\begin{equation}
(\d F,\omega)_{P} = \int_{\Gamma^2\backslash\H^2} [\d F,\overline{\omega}]=\int_{D_2} \d [F,\overline{\omega}] = \oint_{\partial D_2} [F,\omega]=0.
\end{equation}
Or l'homologie classique de $\H^2$ et donc celle de $D_2$ donne $H_2(\H^2,\Z) = 0$. 
Ainsi on obtient réciproquement pour toute forme $\omega_0$ du noyau, on a $[\omega_0,\omega]$ est une forme exacte sur $D_2$. 
Puis l'invariance par $\Gamma^2$ permet d'obtenir le caractère exacte sur $\H^2$.
Enfin le produit scalaire construit sur $V_{w_1,w_2}$ étant non dégénéré et la forme $\omega$ étant quelconque, on obtient le caractère exact de $\omega_0$ sur chacune des coordonnées:
\begin{equation*}
\omega_0 = \sum_{i_1=0}^{w_1}\sum_{i_2=0}^{w_2} \left(\d F_{i_1,i_2}\right) X_1^{i_1} X_2^{i_2} = \d\left(\sum_{i_1=0}^{w_1}\sum_{i_2=0}^{w_2} F_{i_1,i_2} X_1^{i_1} X_2^{i_2}\right).
\end{equation*}
Donc la forme $\omega_0:\H^2\to V_{w_1,w_2}$ est ainsi exacte.\par

On dispose alors d'un résultat issue du Théorème de décomposition de Hodge (Voir de Rham \cite{Rham}):	
\begin{lem}\label{lem1}
On a la décomposition en somme directe orthogonale:
\begin{equation}
\Omega_{k_1,k_2} = \Omega_{k_1,k_2}^{\harm}\oplus\Omega_{k_1,k_2}^0 =\Omega_{k_1,k_2}^{+}\oplus\Omega_{k_1,k_2}^{-}\oplus\Omega_{k_1,k_2}^0.
\end{equation}
\end{lem}

\begin{proof}
On considère une forme $\omega\in\Omega_{k_1,k_2}$. On commence par en extraire la partie harmonique grâce au produit de Petersson d'ordre $2$.
On considère des bases $\mathcal{B}_1$ et $\mathcal{B}_2$ de $S_{k_1}$ et $S_{k_2}$ constituées de formes propres pour les opérateurs de Hecke.
On sait en particulier qu'ils constituent des bases orthogonales pour le produit scalaire de Petersson classique.
Considérons les projections orthogonales :
\begin{align}
\omega_+ &= \sum_{f_1\in\mathcal{B}_1}\sum_{f_2\in\mathcal{B}_2} \frac{(\omega,\omega_{f_1,f_2})_P}{(f_1,f_1)(f_2,f_2)}\omega_{f_1,f_2}\in \Omega_{k_1,k_2}^+,\\
\omega_- &= \sum_{f_1\in\mathcal{B}_1}\sum_{f_2\in\mathcal{B}_2} \frac{(\omega,\overline{\omega_{f_1,f_2}})_P}{(f_1,f_1)(f_2,f_2)}\overline{\omega_{f_1,f_2}}\in \Omega_{k_1,k_2}^-.
\end{align}

Reste à démontrer que $\omega_0 = \omega-\omega_+-\omega_-$ est une forme exacte pour obtenir la décomposition voulue. 
On démontre tour-à-tour que pour $(f_1,f_2)\in\mathcal{B}_1\times\mathcal{B}_2$, on a: $(\omega_0,\omega_{f_1,f_2})_P=(\omega_0,\overline{\omega_{f_1,f_2}})_P=0$.
La forme est orthogonale à $\Omega_{k_1,k_2}^{\harm}$ donc on obtient $\omega_0\in\Omega_{k_1,k_2}^0$.
\end{proof}

On peut relier ces $2$-formes aux $2$-cycles transverses et en déduire la Proposition \ref{prop1}.
\begin{prop}\label{incljkk}
On dispose de la relation d'orthogonalité pour $k_1,k_2\geq 6$: 
\begin{equation}
M_2^{0}(\H^2,\Z) = M_2^{\tr}(\H^2,\Z)\cap\left(\Omega_{k_1,k_2}\right)^{\bot}.
\end{equation}
On peut en déduire les inclusions des idéaux puis des espaces :
\begin{equation}
\mathcal{J}_2\subset\mathcal{J}(k_1,k_2)\quad\text{ et }\quad \Per_{w_1,w_2}\subset V_{w_1,w_2}[\mathcal{J}_2].
\end{equation}
\end{prop}

\begin{proof}
Pour la démonstration, on utilise la forme bilinéaire non dégénérée entre $\Omega^2$ et $M_2$ définie en (\ref{defacc}).
On peut commencer par considérer une chaîne $C \in M_2^{\tr}(\H^2,\Z)$ annulé par toutes les $2$-formes de $\Omega_{k_1,k_2}$.
Son bord est transverse et s'écrit : $\delta_2 C = h + d + v \in (H+D+V)^0$.
On remarque que l'espace des formes contient au moins les formes d'Eisenstein et donc pour tout $\omega\in\Omega_{k_2}$:
\begin{equation}
\langle \omega_{w_1}^0\wedge \omega, C \rangle = 0 \ra \langle F_{w_1} \omega, d+v \rangle = 0 \ra d+v\in (D+V)^0 \ra h \in H^0.
\end{equation}
De la même manière en utilisant l'hypothèse sur les formes exactes du type $\omega\wedge\omega_{w_2}^0\in\Omega_{k_1,k_2}^0$, on retrouve la cas des $1$-formes. 
On en déduit que $h+d\in (H+D)^0$ puis $v\in V^0$. 
Ainsi on trouve également $d\in D^0$ puis $C\in M_2^{0}(\H^2,\Z)$.\par

Réciproquement, on remarque que pour tout $0\leq j \leq 2$ :
\begin{equation}
\varphi_j( M_2^{pte}(\H,\Z)) \subset \left\{\omega:\H^2\to\C|\varphi_j^*\omega=0\right\}^{\bot}.
\end{equation}
En effet, pour $\omega\in\Omega^2(\H^2,\C)$ telle que $\varphi_j^*\omega=0$ et $C\in M_2^{pte}(\H,\Z)$, on a:
$\langle \omega, \varphi_j C\rangle = \langle \varphi_j^{*} \omega, C\rangle = 0.$
Ceci donne bien $\varphi_j C\in \{\varphi_j^*\omega=0\}^{\bot}$.\par

Par ailleurs, dans $M_2(\H^2,\Z)$, on a l'inclusion:
\begin{equation}
\Ker(\delta_2|M_2(\H^2,\Z))=\Im(\delta_3|M_2(\H^2,\Z)) \subset \left\{\omega:\H^2\to\C|\d\omega=0\right\}^{\bot}.
\end{equation}
En effet, l'homologie de la surface de Poincaré donne $H_2(\H^2,\Z)=0$ car $\H^2$ est de dimension $4$.
Puis, d'après le théorème de Stokes, pour $\omega:\H^2\to \C$ une $2$-forme fermée et $C\in M_3^{pte}(\H^2,\Z)$, on obtient :
$\langle \omega, \delta_3 C\rangle = \langle \d \omega, C\rangle = 0.$
C'est à dire $\delta_3 C \in \{\d\omega=0\}^{\bot}$.\par

Et ainsi, on en déduit l'inclusion :
\begin{equation*}
\Ker(\delta_2)+\sum_{j=0}^2 \varphi_j( M_2^{pte}(\H,\Z))  \subset \{\omega:\H^2\to\C|\varphi_0^*\omega=\varphi_1^*\omega=\varphi_2^*\omega=d\omega=0\}^{\bot}.
\end{equation*}
Cette dernière reste valide après tensorisation par $V_{w_1,w_2}^{\Q}$ car chacune des opérations commutent. En effet, on a :
\begin{align}
\left\langle \sum_{m_1,m_2}\omega_{m_1,m_2} X_1^{m_1}X_2^{m_2}, C\right\rangle &= \sum_{m_1,m_2}\langle\omega_{m_1,m_2},C\rangle X_1^{m_1}X_2^{m_2},\\
 \d\left(\sum_{m_1,m_2}\omega_{m_1,m_2} X_1^{m_1}X_2^{m_2}\right) &= \sum_{m_1,m_2}\d\omega_{m_1,m_2} X_1^{m_1}X_2^{m_2}\text{ et }\\
\varphi_j^*\left(\sum_{m_1,m_2}\omega_{m_1,m_2} X_1^{m_1}X_2^{m_2}\right) &= \sum_{m_1,m_2}\varphi_j^*\omega_{m_1,m_2} X_1^{m_1}X_2^{m_2}.
\end{align}
Donc l'annulation de chacune des coordonnées donne l'orthogonale:
\begin{equation*}
\Ker(\delta_2)+\sum_{j=0}^2 \varphi_j( M_2^{pte}(\H,\Z)) \subset \{\omega:\H^2\to V_{w_1,w_2}|\varphi_0^*\omega=\varphi_1^*\omega=\varphi_2^*\omega=d\omega=0\}^{\bot}.
\end{equation*}
Il reste à introduire l'invariance par $\Gamma^2$. Au niveau des $2$-cycles cela permet d'obtenir:
\begin{equation}
\sum_{\gamma\in\Gamma^2}\psi_{\gamma}\left(\Ker(\delta_2)+\sum_{j=0}^2 \varphi_j( M_2^{pte}(\H,\Z))\right)=M_2^{0}(\H^2,\Z).
\end{equation}
En effet, on utilise la linéarité et l'invariance du noyau $\psi_{\gamma}(\Ker(\delta_2))=\Ker(\delta_2)$ pour trouver la définition (\ref{sumM2tr0}).

Puis du coté des $2$-formes différentielles, on trouve:
\begin{equation}
\left(\bigcap_{\gamma\in\Gamma^2}\left\{\omega:\H^2\to V_{w_1,w_2}|\varphi_0^*\omega=\varphi_1^*\omega=\varphi_2^*\omega=d\omega=0\text{ et }\psi^*\omega=\gamma.\omega\right\} \right)^{\bot} = \left(\Omega_{k_1,k_2}\right)^{\bot}.
\end{equation}\par

Pour $g\in\mathcal{J}_2$, on a : $g\delta_2\tau_2\in M_2^{0}(\H^2,\Z)$ par (\ref{defI2cycle}). 
Or on sait que $\Omega_{k_1,k_2}^{\harm}\subset\Omega_{k_1,k_2}$. 
Donc les orthogonaux donnent : $g\tau_2\in \left(\Omega_{k_1,k_2}\right)^{\bot}\subset \left(\Omega_{k_1,k_2}^{\harm}\right)^{\bot}$.
Par la Proposition \ref{propJkk}, ceci donne bien $g\in\mathcal{J}({k_1,k_2})$.\par

Pour un idéal à gauche $I\subset\Z[\Gamma^2]$, on dispose de l'annulateur:
\begin{equation}
V_{w_1,w_2}[I] = \left\{P\in V_{w_1,w_2}\text{ tel que }g.P=0,\,\forall g\in I\right\}.
\end{equation}
L'inclusion des idéaux permet d'obtenir alors l'inclusion des annulateurs:
\begin{equation}
\Per_{w_1,w_2}\subset V_{w_1,w_2}[\mathcal{J}(k_1,k_2)]\subset V_{w_1,w_2}[\mathcal{J}_2].
\end{equation}
\end{proof}

Une fois prouvé que $\mathcal{J}_2=\mathcal{I}_2$, on obtiendra donc la Proposition \ref{prop1}. 

\subsection{Démonstration du Théorème $1$}

On commence par motiver la construction de $\mathcal{J}(k_1,k_2)$.

\begin{prop}\label{propnew}
L'ensemble $V_{w_1,w_2}^{\Q}[\mathcal{J}(k_1,k_2)]$ est le plus petit $\Q$-espace tel que son extension aux complexes contient les polynômes des bi-périodes. 
\end{prop} 

\begin{proof}
On commence par utiliser le fait que l'application de représentation $\Q[\Gamma^2]\to \mathrm{End}_{\Q}(V_{w_1,w_2}^{\Q})$ est surjective. 
Ceci est une conséquence du Théorème de Burnside \cite{Burnside} et du caractère absolument irréductible de la représentation $\Gamma^2\to GL(V_{w_1,w_2}^{\Q})$. 
En effet $(T^i,T^j).X_1^{w_1}X_2^{w_2}=(X_1+i)^{w_1}(X_2+j)^{w_2}$ est une famille génératrice car le déterminant d'une telle famille finie est le produit tensoriel de deux matrices de Vandermonde après renormalisation par les coefficients binomiaux.\par

Notons alors $E$ le plus petit $\Q$-espace tel que $\Per_{w_1,w_2}\subset E\otimes\C$. 
Il existe alors $g_E\in\Q[\Gamma^2]$ qui représente le projecteur $p_E\in \mathrm{End}_{\Q}(V_{w_1,w_2}^{\Q})$ le long de $E$ sur un supplémentaire. 
En particulier, $V_{w_1,w_2}^{\Q}\left[\langle g_E\rangle\right]=\{P\in V_{w_1,w_2}^{\Q}\text{ tel que }g_E.P=0\}=\Ker(p_E)=E$. 
Quitte à multiplier $g_E$ par un entier suffisamment grand, on peut supposer $g_E\in\Z[\Gamma^2]$ sans modifier l'annulateur.\par
L'élément $g_E$ annule $\Per_{w_1,w_2}$ donc appartient à $\mathcal{J}(k_1,k_2)$. 
Ainsi $\langle g_E\rangle\subset \mathcal{J}(k_1,k_2)$ permet de déduire $V_{w_1,w_2}^{\Q}[\mathcal{J}(k_1,k_2)]\subset V_{w_1,w_2}^{\Q}\left[\langle g_E\rangle\right]=E$.
Or l'extension aux complexes $V_{w_1,w_2}^{\C}[\mathcal{J}(k_1,k_2)]$ contient $\Per_{w_1,w_2}$ et $E$ a été supposé minimal démontrant $E=V_{w_1,w_2}^{\Q}[\mathcal{J}(k_1,k_2)]$.
\end{proof}

Pour obtenir le Théorème \ref{thm1}, il faut désormais ajouter la partie d'Eisenstein d'ordre $2$.
Dans la section suivante, on obtiendra de manière calculatoire les expressions des sous-modules de $V_{w_1,w_2}^{\Q}$:
\begin{align}
V_{w_1,w_2}^{\Q}[I_H]&=\{P_1(X_1)\text{ pour }P_1\in W_{w_1}^{\Q}\}=W_{w_1}^{\Q}\otimes 1,\\
V_{w_1,w_2}^{\Q}[I_D]&=\{P\in V_{w_1,w_2}^{\Q}\text{ tel que }P|_{(1,1)+(S,S)}=P|_{(1,1)+(U,U)+(U^2,U^2)}=0\},\\
\text{et }V_{w_1,w_2}^{\Q}[I_H]&=\{X_1^{w_1}P_2(X_2)\text{ pour }P_2\in W_{w_2}^{\Q}\}=X_1^{w_1}\otimes W_{w_2}^{\Q},
\end{align}
où l'on rappel que $W_w^{\Q}=\{P\in V_w^{\Q}\text{ tel que }P|_{1+S}=P|_{1+U+U^2}\}=V_w^{\Q}[\mathcal{I}_1]$.
Les notations d'annulateurs permettent d'avoir une nouvelle écriture de $E_{w_1,w_2}^{\Q}$, défini en (\ref{defEkk}):
\begin{equation}
E_{w_1,w_2}^{\Q}=V_{w_1,w_2}^{\Q}[I_H]+V_{w_1,w_2}^{\Q}[I_D]+V_{w_1,w_2}^{\Q}[I_V].
\end{equation}
 
On a démontré dans la Proposition \ref{incljkk} que l'idéal à gauche $\mathcal{J}_2$ est inclus dans $\mathcal{J}(k_1,k_2)$. 
D'autre part, la Proposition \ref{inclhvd} montre que $\mathcal{J}_2$ est inclus dans les idéaux $I_H,I_D$ et $I_V$ par construction. 
On en déduit l'inclusion des $\Q$-espaces:
\begin{equation}\label{inclQ}
V_{w_1,w_2}^{\Q}[\mathcal{J}(k_1,k_2)]+E_{w_1,w_2}^{\Q}\subset V_{w_1,w_2}^{\Q}[\mathcal{J}_2].
\end{equation}\par

Une conséquence de la Proposition \ref{propnew} est alors que $V_{w_1,w_2}^{\Q}[\mathcal{J}(k_1,k_2)]+E_{w_1,w_2}^{\Q}$ est le plus petit $\Q$-espace convenant aux conditions du Théorème \ref{thm1}. Il reste à démontrer que l'inclusion (\ref{inclQ}) est une égalité. On recherche l'égalité des dimensions de leurs extensions aux corps des complexes.\par

Pour un $2$-cycle $C\in M_2^{pte}(\H^2,\Z)$, on peut considérer l'application de courant le long de $C$ par : 
\begin{equation}
[C] : V_{w_1,w_2}\to \Hom_{\Q}\left(\Omega_{k_1,k_2},\C\right),\quad
Q\mapsto \left(\omega\mapsto\int_{C}[\omega,Q]\right).
\end{equation}
Alors, pour $Q\in V_{w_1,w_2}$ et $\gamma\in\Gamma^2$, on a:
\begin{equation}
[\gamma C](Q)(\omega) = \int_{\gamma C} [\omega,Q] = \int_{C} [\psi_{\gamma}^*\omega,Q] = \int_{C} [\omega,\gamma^{-1}Q] = [C](\gamma^{-1}Q)(\omega).
\end{equation}

Tout polynôme $Q\in V_{w_1,w_2}$ peut s'écrire $Q=g.X_1^{w_1}X_2^{w_2}$ pour un certain $g\in \C[\Gamma^2]$ d'après le caractère transitif de la représentation. 
Donc $[\tau_2](Q)=[g^{-1}\tau_2](X_1^{w_1}X_2^{w_2})$ et ainsi $[\tau_2](Q)=0$ ssi $g^{-1}\tau_2\in\Omega_{k_1,k_2}^{\bot}$. De plus, on a toujours $g^{-1}\tau_2\in M_2^{\tr}(\H^2,\Z)$.

En particulier, la Proposition \ref{incljkk} et la définition (\ref{defI2cycle}) permet d'obtenir: 
\begin{equation}
\Ker([\tau_2]) = \left\langle g^{-1}(X_1^{w_1}X_2^{w_2}) \text{ pour }g\in\mathcal{J}_2\right\rangle= \widetilde{\mathcal{J}_2} V_{w_1,w_2}.
\end{equation}

De plus $[\tau_2](Q)(\omega)=\left[\int_{\tau_2}\omega,Q\right]$.
Donc l'image de l'application :
\begin{equation}
\{\tau_2\}:\Omega_{k_1,k_2}\to V_{w_1,w_2},\quad \omega\mapsto \int_{\tau_2} \omega,
\end{equation}
vérifie $\Im\{\tau_2\}^{\bot}=\Ker([\tau_2])$ et ainsi $\Im\{\tau_2\}=\left(\widetilde{\mathcal{J}_2} V_{w_1,w_2}\right)^{\bot}$.

Or l'invariance par $\Gamma^2$ du produit scalaire (\ref{invg2}) ainsi que son caractère non dégénéré montre que l'orthogonal de $\widetilde{\mathcal{J}_2}V_{w_1,w_2}$ est $V_{w_1,w_2}[\mathcal{J}_2]$.\par

Ainsi, pour $P\in V_{w_1,w_2}[\mathcal{J}_2]$, il existe $\omega\in\Omega_{k_1,k_2}$ tel $P=\int_{\tau_2} \omega\in \Im\{\tau_2\}$.\par

On peut alors appliquer le Lemme \ref{lem1}, pour écrire la décomposition :
$$\omega = \omega_++\omega_-+\omega_0\in\Omega_{k_1,k_2}^{+}\oplus\Omega_{k_1,k_2}^{-}\oplus\Omega_{k_1,k_2}^0.$$

Ceci permet de décomposer $P=\int_{\tau_2}\omega_+ +\int_{\tau_2}\omega_- +\int_{\tau_2}\omega_0$. 
On obtient par définition $\int_{\tau_2}\omega_+\in\Per_{w_1,w_2}$ et $\int_{\tau_2}\omega_-\in\overline{\Per_{w_1,w_2}}$.
Puis on sait que $\omega_0=\d F$ est exacte donc:
\begin{equation}
\int_{\tau_2}\omega_0 = \int_{\partial\tau_2}F 
= \int_{\varphi_0\tau_1}F - \int_{\varphi_1\tau_1} F + \int_{\varphi_2\tau_1} F.
\end{equation}
Puis on peut remarquer que pour $g\in\Z[\Gamma^2], g.\int_{\varphi_j\tau_1}F = \int_{g\varphi_j\tau_1} F$ car on peut choisir $F:\H^2\to V_{w_1,w_2}$ comme étant $\Gamma^2$-invariante. 
Et ainsi d'après les définitions (\ref{defIH}), (\ref{defIV}) et (\ref{defID}), on obtient :
\begin{equation*}
\int_{\varphi_0\tau_1}F\in V_{w_1,w_2}[I_H], \int_{\varphi_1\tau_1}F\in V_{w_1,w_2}[I_D] \text{ et }\int_{\varphi_2\tau_1}F\in V_{w_1,w_2}[I_V].
\end{equation*}
Et ainsi on obtient $\int_{\tau_2}\omega_0 \in E_{w_1,w_2}$.\par
Ceci démontre l'inclusion réciproque de (\ref{inclQ}) sur $\C$. Mais toutes les applications sont $\Q$-linéaires donc on obtient l'égalité des dimensions sur $\Q$. Démontrant ainsi que sur $\Q$, on a:
\begin{equation}
V_{w_1,w_2}^{\Q}[\mathcal{J}(k_1,k_2)]+E_{w_1,w_2}^{\Q}= V_{w_1,w_2}^{\Q}[\mathcal{J}_2].
\end{equation}


\section{Calcul de l'idéal des relations $\mathcal{I}_2$}

L'idéal des relations de Manin $\mathcal{I}_1=\langle 1+S,1+U+U^2\rangle$ peut être introduit via une description homologique. On dispose de $\tau_1$ la $1$-chaîne de $\H$ liant les pointes $\infty$ à $0$. Elle peut être rapprochée du symbole modulaire fondamental $\{\infty,0\}$, voir Mazur \cite{Maz73}. On a alors :
\begin{equation}
\mathcal{I}_1=\{g\in\Z[\Gamma]\text{ tel que }g.\delta_1\tau_1=0\}.
\end{equation}


\subsection{Calcul de $I_H,I_V$ et $I_D$}

Déterminons $I_H,I_V$ et $I_D$ par calcul direct.\par

\begin{prop}
Ces idéaux sont de type fini et on a:
\begin{align}
I_H&=\Big\langle (1+S,1),(1+U+U^2,1),(1,1-T)\Big\rangle,\\
I_V&=\Big\langle (1,1+S),(1,1+U+U^2),(1-US,1)\Big\rangle,\\
\text{et }I_D&=\Big\langle (1,1)+(S,S),(1,1)+(U,U)+(U^2,U^2)\Big\rangle.
\end{align}
\end{prop}

\begin{proof}
Commençons par le calcul de $I_H$. Soit $\sum_i \lambda_i(\gamma_1^i,\gamma_2^i)\in\Z[\Gamma^2]$. Son action sur $H_{id}(\tau_1)$ s'exprime ainsi:
$$\sum_i \lambda_i(\gamma_1^i,\gamma_2^i). H_{id}(\tau_1)=\sum_i \lambda_i H_{\gamma_2^i}(\gamma_1^i.\tau_1)\text{ par (\ref{acth})}.$$
Supposons cet élément dans $I_H$ alors:
$$\sum_{g\in\Gamma/\Gamma_{\infty}}\sum_{\gamma_2^i\in g\Gamma_{\infty}}\lambda_i H_{g}(\gamma_1^i.\tau_1)
=\sum_{g\in\Gamma/\Gamma_{\infty}}H_{g}\left(\sum_{\gamma_2^i\in g\Gamma_{\infty}} \lambda_i \gamma_1^i. \tau_1\right)\in H^0$$
Les espaces $H_g(M_1(\H,\Z))$ sont en somme directe d'après (\ref{sumh}). Donc pour chaque orbite $g\Gamma_{\infty}\in\Gamma/\Gamma_{\infty}$, on a:
$$H_{g}\left(\sum_{\gamma_2^i\in g\Gamma_{\infty}} \lambda_i \gamma_1^i. \tau_1\right)\in H^0,$$
et ainsi, on obtient:
$$\delta_1\sum_{\gamma_2^i\in g\Gamma_{\infty}} \lambda_i \gamma_1^i. \tau_1=\sum_{\gamma_2^i\in g\Gamma_{\infty}} \lambda_i \gamma_1^i. \delta_1\tau_1=0.$$

On déduit alors $\sum_{\gamma_2^i\in g\Gamma_{\infty}} \lambda_i \gamma_1^i\in\mathcal{I}_1=(1+S,1+U+U^2)$.
De plus, le stabilisateur de la pointe $\infty$ dans $\Gamma$ est:
$$\Gamma_{\infty}=<T^a>_{a\in\Z}\text{ où }T=\left(\begin{smallmatrix}1 & 1\\ 0 & 1 \end{smallmatrix}\right)=U^2S.$$
On obtient ainsi:
\begin{equation*}
I_H=\Big\langle (1,1)+(S,T^a),(1,1)+(U,T^a)+(U^2,T^b)\Big\rangle_{a,b\in\Z}
\end{equation*}

Cet idéal est en faite de type fini. Les relations de Manin agissant sur la première coordonnée et la stabilité de la pointe $\infty$ sur la seconde, on peut les scinder suivant:
$$(1,1)+(S,T^a)=(1+S,1)-(S,1-T^a),$$
$$\text{et }(1,1)+(U,T^a)+(U^2,T^b)=(1+U+U^2,1)-(U,1-T^a)-(U^2,1-T^b).$$
Et donc les éléments $(1+S,1)$, $(1+U+U^2,1)$ et $(1,1-T)$ engendrent $I_H$.\par

De la même manière, on réduit le calcul de $I_V$ à celui du stabilisateur de $0$:
$$\Gamma_0=<(US)^a>_{a\in\Z}\text{ où }US=\left(\begin{smallmatrix}1 & 0\\ 1 & 1 \end{smallmatrix}\right).$$
On en déduit:
\begin{equation*}
I_V=\Big\langle (1,1)+((US)^a,S),(1,1)+((US)^a,U)+((US)^b,U^2)\Big\rangle_{a,b\in\Z}
\end{equation*}
Une réécriture de cette famille de ces générateurs donne de la même manière le résultat.\par

Pour le calcul de $I_D$, on rappel que $(\gamma_1,\gamma_2).D_I(c)=D_{\gamma_2\gamma_1^{-1}}(\gamma_1.c)$.
On peut alors réduire l'action de $\Z[\Gamma^2]$ à celle de $\Z[\Gamma]$ via:
$$\sum \lambda_i (\gamma_1^i,\gamma_2^i).D_I(\tau_1)=\sum_g D_g\left(\sum_{g=\gamma_2\gamma_1^{-1}}\lambda_i \gamma_1^i.\tau_1\right).$$

Il suffit dans ce cas de faire agir de manière diagonale $\mathcal{I}_1$ pour obtenir le résultat.
\end{proof}

\subsection{Homologie relative de $PSL_2(\Z)^2$}

Dans la théorie de Manin-Drinfeld \cite{Ma2,Dr1} d'une forme modulaire, on pouvait oublier l'espace $\H$ en étudiant uniquement l'action de $\Gamma$ sur les pointes $\pte$. On rappelle brièvement cette construction pour ensuite la généraliser au cas de l'action de $\Gamma^2$ sur $\H^2$.\par

On définit le groupe $\Z[\pte]^0$ des $1$-chaînes fermées comme étant le noyau de l'application:
$$\Z[\pte]\to\Z,\quad \sum_i \lambda_i [a_i]\mapsto \sum_i \lambda_i.$$
Le Théorème de Manin donne alors la surjectivité de l'application:
\begin{equation}\label{Theta1}
\Theta_1:\Z[\Gamma]\to\Z[\pte]^0,\quad \gamma\to [\gamma.\infty]-[\gamma.0].
\end{equation}
Le noyau de cette application est l'idéal $\mathcal{I}_1$ des relations de Manin.

\subsubsection{Construction du groupe formel des $m$-chaînes de $\pte^2$}

Notons $P_2=\pte^2$ les sommets de $\H^2$. 
L'idée est de traduire la correspondance donnée, pour tout entier $m> 0$, par:
\begin{equation}
M_m^{pte}(\H^2,\Z)\to\Z[P_2^{m+1}],\quad C\mapsto \left(C(e_0),...,C(e_m)\right).
\end{equation}
Pour $m\in\Z_{>0}$, on définit l'application de bord:
$$\delta_{m}:\Z[P_2^{m+1}]\to\Z[P_2^m],\quad \left[(p_1,q_1),...,(p_{m+1},q_{m+1})\right]\mapsto\sum_{j=1}^{m+1}(-1)^j\left[...,(p_{j-1},q_{j-1}),(p_{j+1},q_{j+1}),...\right].$$
On a $\delta_{m}\circ\delta_{m+1}=0$ ainsi que l'exactitude de la suite longue:
$$...\stackrel{\delta_{m+1}}{\longrightarrow}\Z[P_2^{m+1}]\stackrel{\delta_{m}}{\longrightarrow}\Z[P_2^m]\stackrel{\delta_{m-1}}{\longrightarrow}...
\stackrel{\delta_1}{\longrightarrow}\Z[P_2]\stackrel{\delta_{0}}{\longrightarrow}\Z\longrightarrow 0.$$
Ceci permet de définir, pour tout entier $m\geq 0$, le sous-groupe:
\begin{equation}
\Z[P_2^{m+1}]^0=\Ker(\delta_{m})=\Im(\delta_{m+1}).
\end{equation}

Pour $0\leq j\leq 2$, considérons la famille d'applications $\varphi_j^g:\H\to\H^2$ généralisant les applications définies en (\ref{applitrans}) et dépendant d'un paramètre $g\in\Gamma$ par: 
\begin{equation}
\varphi_0^g(z)=(g.\infty,z),\quad\varphi_1^g(z)=(z,g.z)\text{ et }\varphi_2^g(z)=(z,g.0).
\end{equation}

On a vu que les cycles contenus dans les sous-espaces transverses, image de $\varphi_j^g$, amènent des annulations. Nous allons donc quotienter par ces cycles.
Ce sont ceux de la forme:
$$\varphi_j^g(c),\text{ pour }c\in\Z[P_1^{m+1}]^0,j=0,1,2\text{ et }g\in\Gamma.$$
Posons $\Z[P_2^{m+1}]_j^0=\sum_{g\in\Gamma}\varphi_j^g(\Z[P_1^{m+1}]^0)\subset\Z[P_2^{m+1}]^0$.\par
On définit alors le groupe quotient des $m$-chaînes formelles:
\begin{equation}
H_m(P_2)=\Z[P_2^{m+1}]^0/\left(\sum_{j=0}^2 \Z[P_2^{m+1}]^0_j\right).
\end{equation}

\begin{prop}
L'action de $\Gamma^2$ sur $P_2$ se transmet diagonalement aux groupes $\Z[P_2^{m+1}]$. Cette action passe alors au quotient sur $H_m(P_2)$.\par
\end{prop}

\begin{proof}
Pour définir l'action de $\Gamma^2$ sur $H_m(P_2)$, il faut vérifier que l'action sur $\Z[P_2^{m+1}]$ fixe les espaces $\Z[P_2^{m+1}]^0$ et $\Z[P_2^{m+1}]^0_j$ pour tout $j$.
Or on peut mettre en évidence les relations:
$$\delta_{m}\left[(\gamma_1,\gamma_2).(p_1,q_1),...,(\gamma_1,\gamma_2)(p_{m+1},q_{m+1})\right]=(\gamma_1,\gamma_2)\delta_{m}\left[(p_1,q_1),...,(p_{m+1},q_{m+1})\right],$$
$$\text{ et }\varphi_j^g(\gamma z)=
\begin{cases}
(g \infty,\gamma z)=(g,\gamma).\varphi_0^{Id}(z)&\text{ si }j=0,\\
(\gamma z,g \gamma z)=(\gamma,g\gamma).\varphi_1^{Id}(z)&\text{ si }j=1,\\
(\gamma z,g 0)=(\gamma,g).\varphi_2^{Id}(z)&\text{ si }j=2,
\end{cases}$$
qui démontrent la stabilité des sous-groupes.
\end{proof}

\subsubsection{Le triangle fondamental}

On construit un élément $T_2$ de $\Z[P^3_2]$ associé à $\tau_2$ et donné par:
\begin{align}
T_2&=[(\infty,\infty),(0,\infty),(0,0)],\\
\text{et donc : }\delta_2 T_2&=[(0,\infty),(0,0)]-[(\infty,\infty),(0,0)]+[(\infty,\infty),(0,\infty)].\nonumber
\end{align}
On notera $\partial T_2$ la classe de $\delta_2 T_2$ dans $H_1(P_2)$.

Définissons le morphisme de $\Z[\Gamma^2]$-modules:
$$\Theta_2:\Z[\Gamma^2]\to H_1(P_2),\quad(\gamma_1,\gamma_2)\mapsto (\gamma_1,\gamma_2).\partial T_2.$$

\begin{prop}\label{proptheta2}
Le noyau de $\Theta_2$ est l'idéal à gauche $\mathcal{J}_2$.\\
Son image, que nous noterons $H_1(P_2)^0$, est l'ensemble des classes des chaînes engendrées par celles de la forme:
$$\delta_2[(a,g.a),(a,g.b),(b,g.b)]\text{ où }a,b\in\pte\text{ et }g\in\Gamma.$$
\end{prop}

\begin{proof}
Pour calculer l'image de l'application, on choisit un élément qui admet pour représentant:
$$\delta_2[(a,g.a),(a,g.b),(b,g.b)]\text{ avec }a,b\in\pte,g\in\Gamma.$$
Cet élément est l'image de la multiplication par $(1,g)$ de:
$$\delta_2[(a,a),(a,b),(b,b)]$$
Puis on écrit comme dans le cas classique une suite finie de matrice $\gamma_i\in\Gamma$ telle que:
$$b=\gamma_1.\infty\to\gamma_1.0=\gamma_2.\infty\to...\to\gamma_{N-1}.0=\gamma_N.\infty\to\gamma_N.0=a.$$
Ceci se fait de manière constructive grâce aux fractions continues, voir Merel \cite{Merel1991}.\par
Il nous suffit alors de montrer le résultat de décomposition suivante pour $N=2$:
$$\delta_2[(\gamma_2.0,\gamma_2.0),(\gamma_1.\infty,\gamma_2.0),(\gamma_1.\infty,\gamma_1.\infty)]=[(\gamma_1,\gamma_1)+(\gamma_2,\gamma_2)+(\gamma_2,\gamma_1)+(\gamma_2 S,\gamma_1 S)].\partial T_2$$
En effet, il suffit de simplifier le terme de droite qui est la classe dans $H_1(P_2)$ du bord d'un triangle qui se décompose en quatre triangles:

\begin{tabular}{ll}
& \multirow{9}*{
\setlength{\unitlength}{0.5cm}
\begin{picture}(9,7)
\put(1,1){\line(1,0){6}}
\put(1,1){\line(1,1){6}}
\put(4,1){\line(0,1){3}}
\put(4,1){\line(1,1){3}}
\put(7,1){\line(0,1){6}}
\put(4,4){\line(1,0){3}}

\put(0.4,0){$\gamma_1i\infty$}
\put(2.7,0){$\gamma_1 0=\gamma_2\infty$}
\put(6.7,0){$\gamma_2 0$}

\put(8,1){$\gamma_1 \infty$}
\put(7.5,4){$\gamma_1 0=\gamma_2 \infty$}
\put(8,7){$\gamma_2 0$}

\end{picture}}\\
$\delta_2[(\gamma_1.\infty,\gamma_1.\infty),(\gamma_1.0,\gamma_1.\infty),(\gamma_1.0,\gamma_1.0)]$&\\
&\\
$+\delta_2[(\gamma_2.\infty,\gamma_2.\infty),(\gamma_2.0,\gamma_2.\infty),(\gamma_2.0,\gamma_2.0)]$& \\
&\\
$+\delta_2[(\gamma_2.\infty,\gamma_1.\infty),(\gamma_2.0,\gamma_1.\infty),(\gamma_2.0,\gamma_1.0)]$& \\
&\\
$+\delta_2[(\gamma_2S.\infty,\gamma_1S.\infty),(\gamma_2S.0,\gamma_1S.\infty),(\gamma_2S.0,\gamma_1S.0)]$.& \\
&\\
\end{tabular}

Ce résultat peut s'itérer sur la famille des $\gamma_1,...,\gamma_N$ et on obtient bien alors un élément de $\Z[\Gamma^2]\partial T_2$ l'image de $\Theta_2$.\par

Pour démontrer que le noyau de $\Theta_2$ est $\mathcal{J}_2$, on utilise l'expression de la Proposition \ref{inclhvd}.
On vérifie que chacun de ces idéaux $I_H, I_D$ et $I_V.$ annulent tour à tour les côtés de :
$$\delta_2 T_2=[(\infty,\infty),(0,\infty)]-[(\infty,\infty),(0,0)]+[(0,\infty),(0,0)],$$
modulo $\sum_{j=0}^2 \Z[P_2^2]^0_j$.
\end{proof}

L'application $\Theta_2$ n'est pas surjective et on peut déterminer plus précisément $H_1(P_2)^0$:
\begin{prop}\label{propimgHP}
Le sous-groupe $H_1(P_2)^0$ est celui des $1$-chaînes transverses:
\begin{align}
H_1(P_2)^0&=\left[\left(\sum_{j,g} \varphi_j^g(\Z[P_1^2])\right)\cap \Z[P_2^2]^0\right]/\sum_{j,g} \varphi_j^g(\Z[P_1^2]^0)\\
&\cong (H+D+V)^0/(H^0+D^0+V^0).\nonumber 
\end{align}
\end{prop}

\begin{proof}
On commence par montrer que les générateurs obtenus dans la Proposition \ref{proptheta2} sont dans l'espace transverse. En effet, son bord :
$$\delta_2[(a,g.a),(a,g.b),(b,g.b)]=[(a,ga),(a,gb)]-[(a,ga),(b,gb)]+[(a,gb),(b,gb)]\in \Z[P_2^2]^0,$$
vérifie les appartenances :
\begin{align*}
[(a,ga),(a,gb)]&=\varphi_0^{\gamma_1}([ga,gb])\in\varphi_0^{\gamma_1}(\Z[P_1^2]),\\
[(a,ga),(b,gb)]&=\varphi_1^g([a,b])\in\varphi_1^g(\Z[P_1^2]),\\
\text{et }[(a,gb),(b,gb)]&=\varphi_2^{\gamma_2}([a,b])\in\varphi_2^{\gamma_2}(\Z[P_1^2]).
\end{align*}
On a introduit pour cela $\gamma_1,\gamma_2\in\Gamma$ tels que $\gamma_1\infty=a$ et $\gamma_2 0= gb$.\par

Pour démontrer l'inclusion réciproque, on donne un algorithme de décomposition de l'espace $H_1(P_2)$ en triangles simples.\par
On prend un représentant quelconque d'une classe de $H_1(P_2)$:
$$\delta_2[(a_1,a_2),(b_1,b_2),(c_1,c_2)]\in\Z[P_2^2]^0.$$
Il dépend de six pointes $a_1,a_2,b_1,b_2,c_1,c_2\in\pte$.

On va trianguler pour obtenir une combinaison linéaire d'éléments dépend de quatre pointes $\alpha_1,\alpha_2,\beta_1,\beta_2\in\pte$ et de la forme:
$$\delta_2[(\alpha_1,\alpha_2),(\alpha_1,\beta_2),(\beta_1,\beta_2)].$$

Pour cela, on utilise le fait que $\delta_2\circ\delta_3=0$ en ajoutant un nouveau couple de pointes.
Si on adjoint le couple $(a_1,b_2)$, on obtient:
\begin{multline*}
0=\delta_2\circ\delta_3[(a_1,a_2),(a_1,b_2),(b_1,b_2),(c_1,c_2)]\\
=\delta_2[(a_1,b_2),(b_1,b_2),(c_1,c_2)]-\delta_2[(a_1,a_2),(b_1,b_2),(c_1,c_2)]\\
+\delta_2[(a_1,a_2),(a_1,b_2),(c_1,c_2)]-\delta_2[(a_1,a_2),(a_1,b_2),(b_1,b_2)].
\end{multline*}
Ceci permet d'obtenir une réécriture de $\delta_2[(a_1,a_2),(b_1,b_2),(c_1,c_2)]$
$$=\delta_2[(a_1,b_2),(b_1,b_2),(c_1,c_2)]+\delta_2[(a_1,a_2),(a_1,b_2),(c_1,c_2)]-\delta_2[(a_1,a_2),(a_1,b_2),(b_1,b_2)].$$
Le troisième terme est de la forme voulue et les deux premiers ne dépendent plus que de cinq pointes.\par 
Ils se réduisent tous les deux de la même manière. Par exemple, on peut ajouter $(c_1,b_2)$, pour calculer $\delta_2[(a_1,b_2),(b_1,b_2),(c_1,c_2)]$ 
$$=\delta_2[(c_1,b_2),(b_1,b_2),(c_1,c_2)]+\delta_2[(a_1,b_2),(c_1,b_2),(c_1,c_2)]-\delta_2[(a_1,b_2),(c_1,b_2),(b_1,b_2)].$$
Le premier et deuxième termes sont de la forme attendue. Le dernier est en faite un triangle à bord transverse:
$$\delta_2[(a_1,b_2),(c_1,b_2),(b_1,b_2)]=\varphi_2^{\gamma_2}(\delta_2[a_1,c_1,b_1])\in\Z[P_2^2]^0_2,\text{ avec }\gamma_2 0=b_2$$
Sa classe est donc nulle dans $H_1(P_2)$.\par

Ainsi $H_1(P_2)$ est engendré par:
$$\delta_2[(\alpha_1,\alpha_2),(\alpha_1,\beta_2),(\beta_1,\beta_2)]
=\delta_2[(\alpha_1,g_2g_1^{-1}.\alpha_1),(\alpha_1,g_2g_1^{-1}.b),(g_2g_1^{-1}.c,g_2g_1^{-1}.b)],$$
où $\alpha_1=g_1.\infty$, $\alpha_2=g_2\infty$, $b=g_1g_2^{-1}\beta_2$ et $c=g_1g_2^{-1}\beta_1$.\par
Donc les cycles de $H_1(P_2)$ peuvent être décomposés comme combinaison linéaire de cycle de la forme:
$$\delta_2[(a,g.a),(a,g.b),(c,g.b)]\text{ où }a,b,c\in\pte\text{ et }g\in\Gamma.$$

On peut désormais compléter la démonstration de la Proposition \ref{propimgHP}. Supposons que la $1$-chaîne obtenue appartient à $H_1(P_2)^0$ alors 
$\delta_2[(a,ga),(a,gb),(c,gb)]$ a deux de ses côtés dans l'espace transverse imposant le troisième $[(a,ga),(c,gb)]\in\sum_{j,g}\varphi_j^g(\Z[P_1^2])$ c'est à dire $c\in\{a,b\}$. Ceci donne bien l'inclusion manquante.
\end{proof}

\begin{rem}
Les propositions \ref{proptheta2} et \ref{propimgHP} se résument dans la suite exacte décomposant $\Z[\Gamma^2]$ :
\begin{equation}\label{suiexact2}
0\to\mathcal{J}_2\to\Z[\Gamma^2]\stackrel{\Theta_2}{\to}H_1(P_2)^0\to 0.
\end{equation}
\end{rem}

\subsection{Les relations de Manin d'ordre $2$}

\subsubsection{$\mathcal{J}_2$ est de type fini}

Nous montrons que l'idéal $\mathcal{J}_2$ est engendré par des combinaisons linéaires à support dans l'ensemble fini:
$$(U^{i_1},U^{i_2})(S,S)^{\delta}\text{ pour }(i_1,i_2)\in (\Z/3\Z)^2\text{ et }\delta\in \Z/2\Z.$$
Pour cela, nous allons introduire une hauteur sur $H_1(P_2)^0$ associée à une distance de $\pte$ invariante par $\Gamma$.\par

Définissons l'application $d:\pte\times\pte\to \Z_{\geq 0}$ en posant pour $a\neq b\in\pte$:
\begin{multline}
d(a,b)=\min\{N\in\mathbb{N}\text{ tel qu'il existe }\gamma_1,...,\gamma_N\in\Gamma\\
\text{ vérifiant }a=\gamma_1 \infty,\gamma_1 0=\gamma_2 \infty,...,\gamma_N 0=b\},
\end{multline}
et $d(a,a)=0$.

\begin{prop}\label{propdist}
L'application $d$ est une distance invariante par $\Gamma$, au sens où pour tout $a,b,c\in\pte$ et $\gamma\in \Gamma$:
\begin{align*}
d(a,b)=d(b,a),\quad &d(a,b)=0 \Leftrightarrow a=b,\\ 
d(a,c)\leq d(a,b)+d(b,c)\quad\text{ et }\quad &d(\gamma a,\gamma b)=d(a,b).
\end{align*}
\end{prop}

\begin{rem}
1) Pour la suite, on appellera \textit{chaîne} une telle suite de matrices et on notera:
\begin{equation}
a=\gamma_1 \infty \stackrel{\gamma_1}{\to}\gamma_1 0\stackrel{\gamma_2}{\to}...\gamma_N \infty\stackrel{\gamma_N}{\to} \gamma_N 0=b.
\end{equation}
On appellera \textit{longueur} de la chaîne le nombre $N$ de matrices. Les chaînes de longueurs minimales reliant deux points $a$ et $b$ déterminent ainsi la distance $d(a,b)$.\par
2) L'application $d$ est la distance associée au graphe orienté $\mathcal{G}_1$ dont l'ensemble des sommets est $\pte$ et l'ensemble des arrêtes est $PSL_2(\Z)$. En effet, pour tout couple $\big((p_1:q_1),(p_2:q_2)\big)\in\pte^2$, on peut imposer $p_1,p_2,q_1,q_2\in\Z$ et quitte à changer $(p_1:q_1)=(-p_1:-q_1)$, on a:
\begin{equation}
d\big((p_1:q_1),(p_2:q_2)\big)=1 \Leftrightarrow \mat{p_1}{p_2}{q_1}{q_2}\in PSL_2(\Z)\Leftrightarrow \mat{p_2}{-p_1}{q_2}{-q_1}\in PSL_2(\Z).
\end{equation}
Ainsi toute paire de sommets de ce type est reliée par les matrices $\mat{p_1}{p_2}{q_1}{q_2}$ et $\mat{p_2}{-p_1}{q_2}{-q_1}$. Nous représentons ceci sur le schéma suivant où chaque trait représente ainsi deux arrêtes orientées aller et retour.

\vspace{0.5cm}
\hbox{
\hspace*{-.1cm}

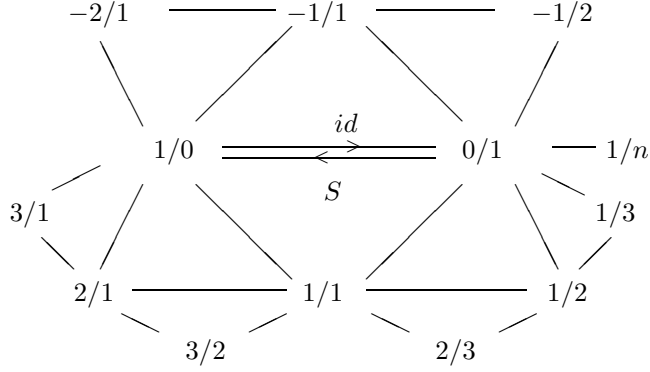
\setlength{\unitlength}{0.7cm}
\begin{picture}(24,7)

\put(12.1,4.5){$id$}
\put(10,4.2){\line(1,0){4}}
\put(12.3,4.2){$_>$}

\put(11.9,3.2){$S$}
\put(10,4){\line(1,0){4}}
\put(11.7,4){$_<$}

\put(8.7,4){${1/0}$}
\put(14.5,4){${0/1}$}
\put(11.2,6.6){${-1/1}$}
\put(11.5,1.3){${1/1}$}
\put(9.5,4.7){\line(1,1){1.8}}
\put(14.5,4.7){\line(-1,1){1.8}}
\put(9.5,3.5){\line(1,-1){1.8}}
\put(14.5,3.5){\line(-1,-1){1.8}}

\put(15.5,3.5){\line(1,-2){0.8}}
\put(8.5,3.5){\line(-1,-2){0.8}}
\put(7.2,1.3){$2/1$}
\put(16.1,1.3){$1/2$}
\put(12.7,1.5){\line(1,0){3}}
\put(8.3,1.5){\line(1,0){2.9}}

\put(7.1,6.6){$-2/1$}
\put(15.8,6.6){$-1/2$}
\put(15.5,4.6){\line(1,2){0.8}}
\put(8.5,4.6){\line(-1,2){0.8}}
\put(12.9,6.8){\line(1,0){2.2}}
\put(9,6.8){\line(1,0){2}}

\put(14,0.2){$2/3$}
\put(13.5,0.7){\line(-2,1){0.7}}
\put(15.2,0.7){\line(2,1){0.7}}
\put(9.3,0.2){$3/2$}
\put(8.8,0.7){\line(-2,1){0.7}}
\put(10.5,0.7){\line(2,1){0.7}}

\put(17,2.8){$1/3$}
\put(17.3,2.5){\line(-1,-1){0.6}}
\put(16.8,3.3){\line(-2,1){0.8}}
\put(6,2.8){$3/1$}
\put(6.6,2.5){\line(1,-1){0.6}}
\put(6.8,3.4){\line(2,1){0.9}}

\put(17.2,4){$1/n$}
\put(17,4.2){\line(-1,0){0.8}}

\end{picture}
}

3) La valence (i.e. le nombre de voisins) de chaque sommet est infinie. Les propriétés de distance données par la Proposition \ref{propdist} se traduisent sur le graphe par sa connexité et son invariance sous l'action de $\Gamma$.
\end{rem}

\begin{proof}
Tout d'abord cette application est bien définie car le développement en fraction continue donne l'existence d'une chaîne de longueur finie, obtenue par l'astuce de Manin \cite{Ma2}, liant $0$ à toute pointe de $\pte$.\\
La condition de séparation est purement formelle, les chaînes de longueur $0$ sont $a=b$.\\
Soient $a,b\in\pte$. Si $d(a,b)=N$ alors il existe une chaîne minimale: $a\stackrel{\gamma_1}{\to}...\stackrel{\gamma_N}{\to}b$. Elle nous permet de construire la chaîne:
$b=\gamma_N S \infty\stackrel{\gamma_N S}{\to}...\stackrel{\gamma_1 S}{\to} \gamma_1 S 0=a.$\\
Ceci démontre que $d(b,a)\leq d(a,b)$ et donc on obtient l'égalité en symétrisant.\\
Soient $a,b,c\in\pte$. On peut concaténer deux chaînes: 
$a\stackrel{\gamma_1}{\to}...\stackrel{\gamma_{d(a,b)}}{\to}b\stackrel{\gamma'_{1}}{\to}...\stackrel{\gamma'_{d(b,c)}}{\to}c.$\\
On obtient une chaîne de longueur $d(a,b)+d(b,c)$ et ainsi $d(a,c)\leq d(a,b)+d(b,c)$.\\
Soit $g\in\Gamma$. Une chaîne de longueur minimale: $a\stackrel{\gamma_1}{\to}...\stackrel{\gamma_{d(a,b)}}{\to}b$ se translate en:
$g a\stackrel{g \gamma_1}{\to}...\stackrel{g \gamma_{d(a,b)}}{\to}g b.$
Ceci donne $d(g a,g b)\leq d(a ,b)$. Puis on applique ceci à $g^{-1}\in\Gamma$ et $ga,gb\in\pte$ pour obtenir l'inégalité manquante.
\end{proof}

Ceci permet de décomposer l'espace image de $\Theta_1:\Z[\Gamma]\to \Z[\pte]^0$, définie en (\ref{Theta1}), suivant une hauteur:
\begin{equation}
h:\pte\to\Z_{\geq 0}, a\mapsto \max(d(a,\infty),d(a,0)),
\end{equation}
en posant, pour tout entier $M>0$:
\begin{equation}
\Z[\pte]^0_M=\left\{\sum_a \lambda_a [a]\in\Z[\pte]^0\text{ tel que }\lambda_a\neq 0\Rightarrow h(a)\leq M\right\}.
\end{equation}

Définissons désormais les parties de $\pte$:
\begin{align*}
B_0&=\{\infty,0,1,-1\}=\{a\in\pte\text{ tel que }h(a)\leq 1\},\\
B_1&=\{(p:q)\in\pte\text{ tel que }0<p<q\},\\
B_2&=\{(p:q)\in\pte\text{ tel que }0<q<p\},\\
B_3&=\{(p:q)\in\pte\text{ tel que }0<-p<q\},\\
\text{et }B_4&=\{(p:q)\in\pte\text{ tel que }0<q<-p\}.
\end{align*}

Ceci nous permet d'obtenir la partition suivante de $\pte$:
\begin{equation}\label{partpte}
\pte=B_0\sqcup B_1\sqcup B_2\sqcup B_3\sqcup B_4.
\end{equation}

De plus, les actions de $S=\mat{0}{-1}{1}{0}$ et $\varepsilon=\mat{-1}{0}{0}{1}$ stabilisent $B_0$ et échangent les espaces $B_1,B_2,B_3$ et $B_4$ selon:
$$B_1=\varepsilon SB_2=\varepsilon B_3=SB_4.$$

Pour tout $j\in\{1,2,3,4\}$, l'ensemble des voisins des sommets de $B_j$ sur le graphe $\mathcal{G}_1$ est $B_j\cup B_0$. On peut ainsi réduire par symétrie l'étude du graphe à l'étude de $B_1\cup B_0$.

\begin{lem}\label{lemtri}
Soit $(p_1:q_1),(p_2:q_2)\in B_1$ vérifiant $d\left(\frac{p_1}{q_1},\frac{p_2}{q_2}\right)=1$. Alors il existe exactement deux pointes à distance $1$ de $\frac{p_1}{q_1}$ et de $\frac{p_2}{q_2}$ qui sont:
$$\frac{p_1-p_2}{q_1-q_2}\quad\text{ et }\quad\frac{p_1+p_2}{q_1+q_2}.$$
De plus, leurs hauteurs sont données par:
$$h\left(\frac{p_1-p_2}{q_1-q_2}\right)\leq \min\left[h\left(\frac{p_1}{q_1}\right),h\left(\frac{p_2}{q_2}\right)\right]
\text{ et }h\left(\frac{p_1+p_2}{q_1+q_2}\right)= \min\left[h\left(\frac{p_1}{q_1}\right),h\left(\frac{p_2}{q_2}\right)\right]+1.$$
\end{lem}

\begin{proof}
Posons $\gamma=\mat{p_1}{p_2}{q_1}{q_2}$, afin d'obtenir $\frac{p_1}{q_1}=\gamma \infty$ et $\frac{p_2}{q_2}=\gamma 0$. L'hypothèse $d\left(\frac{p_1}{q_1},\frac{p_2}{q_2}\right)=1$ se traduit par $\gamma\in\Gamma$ et ainsi l'invariance de $d$ par $\Gamma$ se traduit par:
\begin{center}
Les pointes à distance $1$ de $\gamma \infty$ et $\gamma 0$ sont $\gamma (1)=\frac{p_1+p_2}{q_1+q_2}$ et $\gamma (-1)=\frac{p_1-p_2}{q_1-q_2}$.
\end{center}\par
Pour calculer leurs hauteurs, on commence par remarquer que pour tout $\alpha \in B_1,h(\alpha)=d(\infty,\alpha)$ et ainsi la pointe qui nous concerne après translation par $\gamma$ provient de $\gamma (\frac{-q_2}{q_1})=\infty$. Or on sait que $\frac{-q_2}{q_1}\in B_3\cup B_4\cup \{-1\}$. Lorsque $\frac{-q_2}{q_1}=-1$ le résultat est vérifiable simplement car $d(-1,\infty)=d(-1,0)=d(-1,1)-1$. Sinon, par symétrie entre $p_1/q_1$ et $p_2/q_2$, on peut supposer $\frac{-q_2}{q_1}\in B_3$ et on obtient:
$$d(\frac{-q_2}{q_1},-1)\leq d(\frac{-q_2}{q_1},0)\leq d(\frac{-q_2}{q_1},\infty)\leq d(\frac{-q_2}{q_1},1).$$
La première inégalité se traduit par $h\left(\frac{p_1-p_2}{q_1-q_2}\right)\leq h\left(\frac{p_1}{q_1}\right)$ après translation par $\gamma$. Par l'absurde, on montre que parmi la deuxième et la troisième inégalité au moins une est stricte faute de contredire le début du Lemme. Ainsi on obtient:
$$d(\frac{-q_2}{q_1},0)<d(\frac{-q_2}{q_1},1)\leq d(\frac{-q_2}{q_1},0) + d(0,1)=d(\frac{-q_2}{q_1},0)+1.$$
Ceci se translate par $\gamma$ en $h\left(\frac{p_1+p_2}{q_1+q_2}\right)= h\left(\frac{p_1}{q_1}\right)+1$.
\end{proof}

Nous avons désormais les outils pour montrer que $\mathcal{J}_2$ est de type fini. 
Commençons par appliquer la méthode à l'idéal $\mathcal{I}_1$ des relations de Manin.
L'idéal $\mathcal{I}_1$ est le noyau de l'application $\Theta_1:\Z[\Gamma]\to \Z[\pte]^0$. Un élément du groupe $\Gamma$ est une arrête orientée de $\mathcal{G}_1$ dont les extrémités dans $\pte$ sont données par l'application $\Theta_1$. Ainsi un élément du groupe $\Z[\Gamma]$ correspond à un $1$-cycle du graphe $\mathcal{G}_1$ et son image par $\Theta_1$ correspond à son bord. Les chemins de $\mathcal{I}_1$, c'est-à-dire annulant $\Theta_1$, sont ainsi les chemins fermées de $\mathcal{G}_1$.\par

Pour un élément $\sum \lambda_{\gamma} [\gamma]\in\Z[\Gamma]$, définissons son support la partie de $\pte$ donnée par: 
$$Supp(\sum \lambda_{\gamma} [\gamma])=\bigcup_{\lambda_{\gamma}\neq 0} \{\gamma \infty,\gamma 0\}.$$
Notons $C^0(X)$ pour les $1$-cycles fermées de $\mathcal{G}_1$ à support dans $X$ une partie de $\pte$:
$$C^0(X)=\{g\in\mathcal{I}_1\text{ tel que }Supp(g)\subset X\}.$$
La partition de $\pte$ (\ref{partpte}) et les propriétés d'invariances élémentaires montrent que:
\begin{align*}
\mathcal{I}_1\cong C^0(\pte)&= C^0(B_0\cup B_1) + C^0(B_0\cup B_2) + C^0(B_0\cup B_3) + C^0(B_0\cup B_4)\\
&=C^0(B_0\cup B_1) + \varepsilon S.C^0(B_0\cup B_1) + \varepsilon .C^0(B_0\cup B_1) + S.C^0(B_0\cup B_1).
\end{align*}
Ainsi on recherche les chaînes fermées à support dans $B_0\cup B_1$. Une méthode de descente sur la hauteur donne alors:

\begin{prop}
L'idéal $\mathcal{I}_1$ est engendré comme $\Z[\Gamma]$-module par les chemins fermés de $\mathcal{G}_1$ à support dans $B_0$:
\begin{equation}
\mathcal{I}_1 = \Z[\Gamma] . C^0(B_0)=\Z[\Gamma] (1+S,1+U+U^2).
\end{equation}
\end{prop}

\begin{proof}
Soit $g=\sum \mu_{\gamma} [\gamma] \in \Ker(\Theta_1)$.
Si $g\neq 0$, il existe un point du support $b\in Supp(g)$ de hauteur maximale. Posons:
$$R_b=\sum_{\gamma\in\Gamma\text{ tel que }b\in Supp([\gamma])} \mu_{\gamma} \Theta_1(\gamma)\text{ avec }\Theta_1(\gamma)=\pm([b]-[a_{\gamma}]).$$
On va remplacer $b$ par un de ces voisins de hauteur inférieure pour réduire la hauteur maximale du support. Sans perdre en généralité, d'après l'invariance par $\Gamma$, on peut supposer $b\in B_1$. Pour $\gamma_1,\gamma_2$ distincts vérifiant $b\in Supp(\gamma_j)$, on peut supposer que $\Theta_1(\gamma_1)=b-a_1$ et $\Theta_1(\gamma_2)=b-a_2$ quitte à changer $\gamma$ par $\gamma S$ et $\mu_{\gamma S}=-\mu_{\gamma}$. On a $a_1\neq a_2$ donc $d(a_1,a_2)\in\{1,2\}$. On distingue alors ces deux cas.\par
Si $d(a_1,a_2)=1$ alors d'après le Lemme \ref{lemtri}, il existe une unique pointe $c$ tel que $b$ et $c$ soit à distance $1$ de $a_1$ et $a_2$. De plus, la hauteur de $b$ étant maximale on obtient $h(c)\leq \min\left[h(a_1),h(a_2)\right]<h(b)$.\par
Si $d(a_1,a_2)=2$ alors la situation est analogue. $a_1=(p_1:q_1)$ est l'image de $1$ et $a_2=(p_2:q_2)$ l'image de $-1$ par la matrice $\gamma=\mat{p_1+p_2}{p_1-p_2}{q_1+q_2}{q_1-q_2}$. On obtient alors $b\in\{\gamma 0,\gamma \infty\}$ car seules deux pointes sont à distance $1$ des pointes $1$ et $-1$ et notons $c$ la seconde différente de $b$. On obtient à nouveau une configuration du type du Lemme \ref{lemtri}. Et ainsi on a: $h(c)+1=\max[h(a_1),h(a_2)]\leq h(b)$.\par
Ainsi dans les deux cas, on a:
$$\Theta_1\left([\gamma_1]-[\gamma_2]\right)=([b]-[a_1])-([b]-[a_2])=([c]-[a_1])-([c]-[a_2])=\Theta_1\left([\gamma_1']-[\gamma_2']\right).$$
Ainsi $[\gamma_1]-[\gamma_2]-[\gamma_1']+[\gamma_2']$ est un élément du noyau. Ces quatre matrices sont à support dans $\{b,a_1,a_2,c\}$ un même translaté de $B_0$. Et comme on a $\sum_{\gamma\in\Gamma\text{ tel que }b\in Supp(\gamma)}\mu_{\gamma}=0$ alors, par itération finie, on peut construire un élément $R_b'$ à support de hauteur inférieure tel que: $R_b-R_b'\in \Z[\Gamma] C^0(B_0)$ et décomposer:
$$g=\sum_{\gamma\in\Gamma\text{ tel que }b\notin Supp[\gamma]}\mu_{\gamma}[\gamma]+R_b'+\left(R_b-R_b'\right).$$
Le support de $g$ étant fini on réduit bien celui-ci à une somme d'éléments de $\Gamma C^0(B_0)$ par récurrence sur la hauteur maximale. On remarquera notamment que dans le déroulement de la démonstration $b$ n'est pas nécessairement l'unique élément de hauteur maximale. Mais l'ensemble de ces pointes est fini et son cardinal diminue strictement car la seule nouvelle pointe introduite vérifie $h(c)<h(b)$.
\end{proof}

Par analogie, il est alors naturel d'introduire le graphe $\mathcal{G}_2$ dont les sommets sont $P_2=\pte^2$ et les arêtes sont $H\cup D\cup V$. Un élément $(\gamma_1,\gamma_2)\in\Gamma^2$ est identifié dans ce graphe par le triangle orienté:
$$[(\gamma_1.\infty,\gamma_2.\infty),(\gamma_1 0,\gamma_2 \infty),(\gamma_1 0,\gamma_2 0)].$$
Son bord est l'image par $\Theta_2$ c'est à dire la somme des trois arêtes orientées dans $H_1(P_2)^0$:
$$[(\gamma_1 \infty,\gamma_2 \infty),(\gamma_1 0,\gamma_2 \infty)]+
[(\gamma_1 0,\gamma_2 \infty),(\gamma_1 0,\gamma_2 0)]+
[(\gamma_1 0,\gamma_2 0),(\gamma_1 \infty,\gamma_2 \infty)].$$\par
Le support d'un élément de $g=\sum \lambda_{\gamma_1,\gamma_2}[\gamma_1,\gamma_2]\in\Z[\Gamma^2]$ est la partie de $\pte^2$ définie par:
$$Supp(g)=\bigcup_{\lambda_{\gamma_1,\gamma_2}\neq 0} \{(\gamma_1 \infty,\gamma_2 \infty),(\gamma_1 0,\gamma_2 \infty),(\gamma_1 0,\gamma_2 0)\}.$$
Le schéma suivant fourni une représentation de la partie de $\mathcal{G}_2$ à support dans $\{\infty,0,1\}^2\subset P_2$. Pour une question de lisibilité nous avons représenté ici huit fois les mêmes neuf couples de pointes de $P_2$ pour illustrer l'ensemble des recouvrement possibles de cette partie. Les couples de matrices représentent alors les différents triangles orientés possibles.

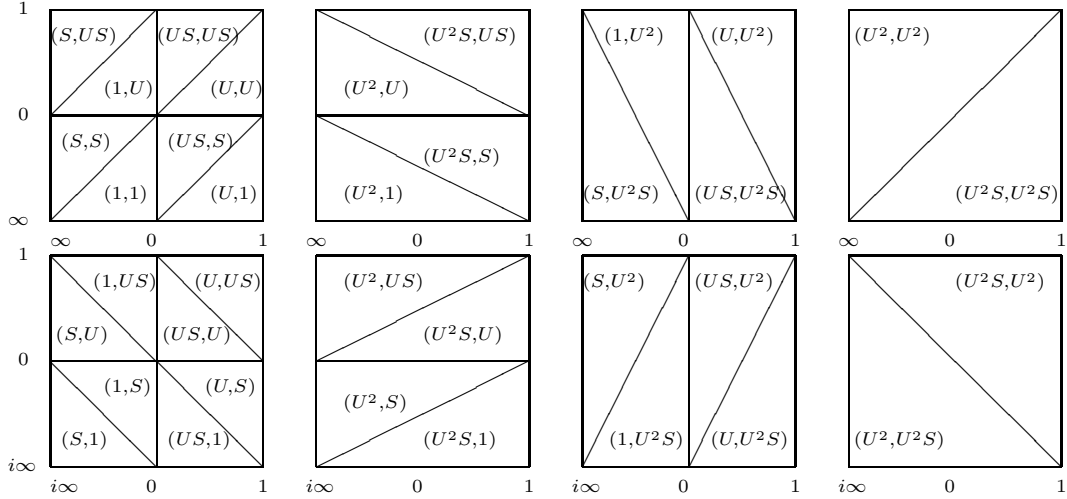
\hbox{
\hspace{1.5cm}

\setlength{\unitlength}{0.7cm}
\begin{picture}(24,5)
\put(0,0){\line(1,0){4}}
\put(0,2){\line(1,0){4}}
\put(0,4){\line(1,0){4}}
\put(0,0){\line(0,1){4}}
\put(2,0){\line(0,1){4}}
\put(4,0){\line(0,1){4}}
\put(0,2){\line(1,1){2}}
\put(0,0){\line(1,1){4}}
\put(2,0){\line(1,1){2}}
\put(1,0.5){$_{(1,1)}$}
\put(0.2,1.5){$_{(S,S)}$}
\put(3,0.5){$_{(U,1)}$}
\put(2.2,1.5){$_{(US,S)}$}
\put(1,2.5){$_{(1,U)}$}
\put(0.0,3.5){$_{(S,US)}$}
\put(3,2.5){$_{(U,U)}$}
\put(2.0,3.5){$_{(US,US)}$}

\put(5,0){\line(0,1){4}}
\put(9,0){\line(0,1){4}}
\put(5,0){\line(1,0){4}}
\put(5,2){\line(1,0){4}}
\put(5,4){\line(1,0){4}}
\put(5,2){\line(2,-1){4}}
\put(5,4){\line(2,-1){4}}
\put(10,0){\line(0,1){4}}
\put(14,0){\line(0,1){4}}
\put(10,0){\line(1,0){4}}
\put(12,0){\line(0,1){4}}
\put(10,4){\line(1,0){4}}
\put(10,4){\line(1,-2){2}}
\put(12,4){\line(1,-2){2}}
\put(15,0){\line(0,1){4}}
\put(19,0){\line(0,1){4}}
\put(15,0){\line(1,0){4}}
\put(15,4){\line(1,0){4}}
\put(15,0){\line(1,1){4}}
\put(5.5,0.5){$_{(U^2,1)}$}
\put(5.5,2.5){$_{(U^2,U)}$}
\put(7,1.2){$_{(U^2S,S)}$}
\put(7,3.5){$_{(U^2S,US)}$}

\put(10.4,3.5){$_{(1,U^2)}$}
\put(12.4,3.5){$_{(U,U^2)}$}
\put(15.1,3.5){$_{(U^2,U^2)}$}
\put(10,0.5){$_{(S,U^2S)}$}
\put(12.1,0.5){$_{(US,U^2S)}$}
\put(17,0.5){$_{(U^2S,U^2S)}$}
\put(0,-0.4){$_{\infty}$}
\put(1.8,-0.4){$_{0}$}
\put(3.9,-0.4){$_{1}$}
\put(4.8,-0.4){$_{\infty}$}
\put(6.8,-0.4){$_{0}$}
\put(8.9,-0.4){$_{1}$}
\put(9.8,-0.4){$_{\infty}$}
\put(11.8,-0.4){$_{0}$}
\put(13.9,-0.4){$_{1}$}
\put(14.8,-0.4){$_{\infty}$}
\put(16.8,-0.4){$_{0}$}
\put(18.9,-0.4){$_{1}$}
\put(-0.8,0){$_{\infty}$}
\put(-0.6,2){$_{0}$}
\put(-0.6,4){$_{1}$}

\end{picture}
}

\bigskip

\hbox{
\hspace{1.5cm}

\setlength{\unitlength}{0.7cm}
\begin{picture}(24,4)
\put(0,0){\line(1,0){4}}
\put(0,2){\line(1,0){4}}
\put(0,4){\line(1,0){4}}
\put(0,0){\line(0,1){4}}
\put(2,0){\line(0,1){4}}
\put(4,0){\line(0,1){4}}

\put(0,4){\line(1,-1){4}}
\put(0,2){\line(1,-1){2}}
\put(2,4){\line(1,-1){2}}

\put(0.2,0.5){$_{(S,1)}$}
\put(1,1.5){$_{(1,S)}$}
\put(2.2,0.5){$_{(US,1)}$}
\put(2.9,1.5){$_{(U,S)}$}
\put(0.1,2.5){$_{(S,U)}$}
\put(0.8,3.5){$_{(1,US)}$}
\put(2.1,2.5){$_{(US,U)}$}
\put(2.7,3.5){$_{(U,US)}$}

\put(5,0){\line(0,1){4}}
\put(9,0){\line(0,1){4}}
\put(5,0){\line(1,0){4}}
\put(5,2){\line(1,0){4}}
\put(5,4){\line(1,0){4}}
\put(5,2){\line(2,1){4}}
\put(5,0){\line(2,1){4}}

\put(10,0){\line(0,1){4}}
\put(14,0){\line(0,1){4}}
\put(10,0){\line(1,0){4}}
\put(12,0){\line(0,1){4}}
\put(10,4){\line(1,0){4}}
\put(10,0){\line(1,2){2}}
\put(12,0){\line(1,2){2}}

\put(15,0){\line(0,1){4}}
\put(19,0){\line(0,1){4}}
\put(15,0){\line(1,0){4}}
\put(15,4){\line(1,0){4}}
\put(15,4){\line(1,-1){4}}

\put(7,0.5){$_{(U^2S,1)}$}
\put(7,2.5){$_{(U^2S,U)}$}
\put(5.5,1.2){$_{(U^2,S)}$}
\put(5.5,3.5){$_{(U^2,US)}$}

\put(10,3.5){$_{(S,U^2)}$}
\put(12.1,3.5){$_{(US,U^2)}$}
\put(17,3.5){$_{(U^2S,U^2)}$}

\put(10.5,0.5){$_{(1,U^2S)}$}
\put(12.4,0.5){$_{(U,U^2S)}$}
\put(15.1,0.5){$_{(U^2,U^2S)}$}

\put(0,-0.4){$_{i\infty}$}
\put(1.8,-0.4){$_{0}$}
\put(3.9,-0.4){$_{1}$}
\put(4.8,-0.4){$_{i\infty}$}
\put(6.8,-0.4){$_{0}$}
\put(8.9,-0.4){$_{1}$}
\put(9.8,-0.4){$_{i\infty}$}
\put(11.8,-0.4){$_{0}$}
\put(13.9,-0.4){$_{1}$}
\put(14.8,-0.4){$_{i\infty}$}
\put(16.8,-0.4){$_{0}$}
\put(18.9,-0.4){$_{1}$}
\put(-0.8,0){$_{i\infty}$}
\put(-0.6,2){$_{0}$}
\put(-0.6,4){$_{1}$}

\end{picture}
}

\bigskip

L'idéal $\mathcal{J}_2$ est le noyau de l'application $\Theta_2:\Z[\Gamma^2]\to H_1(P_2)^0$. Un élément de $\mathcal{J}_2$ correspond donc à une combinaison linéaire de triangles orientés dont les bords orientés se compensent. On notera $C_2^0(X)\subset\Z[\Gamma^2]$ les tels recouvrements de $2$-chaînes fermées à support dans une partie $X\subset P_2$. La décomposition de $\pte$ et le fait que les pointes de $B_0$ sont un point de passage entre deux parties, nous permet d'obtenir:
\begin{equation}
C_2^0(P_2)=\sum_{i_1,i_2=1}^4 C_2^0\left((B_0\cup B_{i_1})\times (B_0\cup B_{i_2})\right).
\end{equation}
On réduit alors ceci par descente sur la hauteur. En effet, les projections suivants les coordonnées respectent la structure du graphe $\mathcal{G}_1$.

\begin{prop}
L'idéal $\mathcal{J}_2$ est engendré comme $\Z[\Gamma^2]$-module par les $2$-chaînes fermées de $\mathcal{G}_2$ à support dans $B_0^2$:
\begin{equation}
\mathcal{J}_2 = \Z[\Gamma^2] . C_2^0(B_0^2).
\end{equation}
\end{prop}

\begin{proof}
Comme $H_1(P_2)^0$ est le sous-groupe des chaînes transverses alors on va traiter une coordonnée puis l'autre par projection. Définissons pour tout couple d'entiers $M_1,M_2>0$, le sous-ensemble des $2$-chaînes bornées par:
\begin{equation*}
H_1(P_2)^0_{(M_1,M_2)}=\left\{C\in H_1(P_2)^0\text{ tel que }Supp(C)\subset (h\times h)^{-1}([0,M_1]\times[0,M_2])\right\}.
\end{equation*}
Donc les relations vérifiées par $H_1(P_2)^0_{(M_1,M_2)}$ sont les translatés par $(\Gamma,1)\subset\Gamma^2$, agissant sur la première coordonnées, de celles vérifiées par $H_1(P_2)^0_{(1,M_2)}$. Puis le même raisonnement sur la seconde coordonnée réduit l'étude de $H_1(P_2)^0_{(1,M_2)}$ à celle de $H_1(P_2)^0_{(1,1)}$. Ce dernier correspond bien à $C_2^0(B_0^2)$.
\end{proof}

\subsubsection{Les générateurs de $\mathcal{I}_2$}

\begin{thm}\label{thm3}
L'idéal $\mathcal{J}_2$ annulateur de $\partial T_2$ dans $\Z[\Gamma^2]$ est:
\begin{multline*}
\mathcal{J}_2=\mathcal{I}_2=\Big\langle (1+S,1+S),(S,S)+(S,US)+(US,US)+(1,U)-(U^2,U^2),\\
(1+U+U^2,1)((1,1)+(S,S)),(1,1+U+U^2)((1,1)+(S,S))\Big\rangle
\end{multline*}
\end{thm}

Le schéma précédent permet d'observer ces annulations des images de $\partial T_2$ dans le graphe $\mathcal{G}_2$. Il nous permet de de proposer les décompositions utiles pour la démonstration. En effet, les recouvrements s'annulent si les segments horizontaux, verticaux et diagonaux se compensent entre eux respectivement. Ces annulations se traduisent par l'appartenance aux idéaux $I_H$, $I_V$ et $I_D$ respectivement.

\begin{proof}
On commence par démontrer que chacun des générateurs appartient bien à $I_H\cap I_D\cap I_V$.
Pour cela, on donne une écriture explicite:
\begin{align*}
(1+S,1+S)&=(1,1+S)(1+S,1)\in I_H\\
&=(1+S,1)(1,1+S)\in I_V\\
&=(1,1+S)[(1,1)+(S,S)]\in I_D 
\end{align*}

Les deux derniers sont dans $I_D$ comme multiple de $(1,1)+(S,S)$ et on a:
\begin{align*}
(1,1+U+U^2)&[(1,1)+(S,S)]=(1,S+US+U^2S)(1+S,1)+(1,1+U+U^2)(1,1-U^2S)\in I_H\\
&=(1,1+U+U^2)+(S,1)(1,1+U+U^2)(1,1+S)-(S,1)(1,1+U+U^2)\in I_V
\end{align*}
et
\begin{align*}
(1+U+U^2,1)&[(1,1)+(S,S)]=(S+US+U^2S,1)(1,1+S)+(1+U+U^2,1)(1,1-US)\in I_V\\
&=(1+U+U^2,1)+(1,S)(1+U+U^2,1)(1+S,1)-(1,S)(1+U+U^2,1)\in I_H.
\end{align*}
Finalement, on décompose $(S,S)+(S,US)+(US,US)+(1,U)-(U^2,U^2)$ en somme de deux éléments de $I_H$ en le découpant selon: 
\begin{align*}
(S,S)+(1,U)&=(1+S,1)+(1,U)(1,1-U^2S)\in I_H,\\
(S,US)+(US,US)-(U^2,U^2)&=(1+U,US)(1+S,1)\\
&-(1,US)(1+U+U^2,1)-(U^2,U^2)(1,1-U^2S)\in I_H.
\end{align*}
C'est aussi une somme d'éléments de $I_V$:
\begin{align*}
(US,US)+(1,U)&=(1,1+S)+(1,U)(1-US,1)\in I_V,\\
(S,S)+(S,US)-(U^2,U^2)&=(S,1+U)(1,1+S)-(S,1)(1,1+U+U^2)-(U^2,U^2)(1-US,1)\in I_V.
\end{align*}
et enfin une somme d'éléments de $I_D$:
\begin{align*}
(S,US)+(1,U)&=(1,U)[(1,1)+(S,S)]\in I_D,\\
(S,S)+(US,US)-(U^2,U^2)&=[(1,1)+(U,U)][(1,1)+(S,S)]-[(1,1)+(U,U)+(U^2,U^2)]\in I_D.
\end{align*}
Or on a réduit $\mathcal{I}_2$ à l'étude des relations vérifiées par les triangles à coordonnées parmi $\{0,\infty,1,-1\}^2$. Puis l'action involutive de $S$ permet de transformer les arrêtes contenant la pointe $-1$ sur une de ses coordonnées en des arrêtes à support dans $\{0,\infty,1\}^2$. De plus les coefficients intervenant dans les combinaisons sont nécessairement bornées car il y a un nombre fini d'arêtes. Une étude exhaustive des combinaisons des $36$ triangles obtenus donne bien l'égalité des idéaux. Cette étude a été confirmée par ordinateur par un calcul exhaustif des combinaisons possibles. D'autre part la surjectivité obtenue dans le section suivante serait un argument suffisant pour obtenir le caractère complet de la famille de générateurs.
\end{proof}

La Proposition \ref{prop1} et Théorème \ref{thm3} fournissent un système exhaustif de relations vérifiées par le polynôme des bi-périodes.


\section{Contrôle de l'espace des polynômes des bi-périodes}

Nous allons ici considérer les structures sur $\Q$. Les espaces étant de dimensions finies et sans torsions, les relations sur $\Q$ peuvent être compiler sur $\Z$.
L'espace des formes modulaires $S_k$ peut en réalité être vu comme le produit tensoriel $S_k^{\Q}\otimes_{\Q}\C$ avec $S_k^{\Q}$ les formes dont les coefficients de Fourrier sont rationnelles.
Ceci s'étend aux $2$-formes et on remarquera que l'espace $\Per_{k_1,k_2}$ dispose également d'une $\Q$-structure.\par

La partie précédente nous permet d'obtenir l'inclusion $\Per_{w_1,w_2}\subset W_{w_1,w_2}$. 
On recherche ici à préciser l'écart qu'il existe dans cette inclusion. 
Nous allons alors en déduire une démonstration du Théorème \ref{thm2}.

\subsection{Calcul de la partie Eisenstein d'ordre $2$}

Pour déterminer $E_{w_1,w_2}^{\Q}$, il nous reste à préciser le calcul de $V_{w_1,w_2}^{\Q}[I_D]$ car on a déjà:
\begin{equation*}
V_{w_1,w_2}^{\Q}[I_H] = W_{w_1}^{\Q}\otimes 1\text{ et }V_{w_1,w_2}^{\Q}[I_V] = X_1^{w_1}\otimes W_{w_2}^{\Q}.
\end{equation*}

\begin{prop}\label{propsection}
L'application $\psi: V_{w_1,w_2}^{\Q}[I_D]\to W_{w_1+w_2}^{\Q}, P(X_1,X_2)\mapsto P(Z,Z)$ est surjective.
On dispose d'une section rationnelle.\\ Un élément :
$P(Z)=\sum_{m=0}^{w_1+w_2}\binom{w_1+w_2}{m}a_{m}(-Z)^{w_1+w_2-m}$ admet pour antécédent :
\begin{equation}
Q(X_1,X_2)=\sum_{m_1=0}^{w_1}\sum_{m_2=0}^{w_2} \binom{w_1}{m_1}\binom{w_2}{m_2}a_{m_1+m_2}(-X_1)^{w_1-m_1}(-X_2)^{w_2-m_2}.
\end{equation}
\end{prop}

\begin{proof}
Par le calcul, on trouve bien $Q\in V_{w_1,w_2}^{\Q}[I_D]$ et $Q(Z,Z)=P(Z)$ impliquant la surjectivité de $\psi$.
\end{proof}

\begin{rem}
L'idée de la formulation précédente provient du Théorème d'Eichler-Shimura qui limite la recherche d'antécédent au polynôme des périodes.
Puis la construction d'un antécédent de $P_f(Z)=\int_{\infty}^0f(z)(X-z)^{w_1+w_2}\d z\in W_{w_1+w_2}$ pour $f\in S_{w_1+w_2+2}$, déterminant la surjectivité de $\psi\otimes_{\Q}\C$, nous inspire ce résultat. Posons :
\begin{align*}
Q_f(X_1,X_2)&=\int_{\infty}^0f(z)(X_1-z)^{w_1}(X_2-z)^{w_2} \d z.\\
=\sum_{m_1=0}^{w_1}\sum_{m_2=0}^{w_2}& \binom{w_1}{m_1}\binom{w_2}{m_2} \int_{\infty}^0 f(z)z^{m_1+m_2}(-X_1)^{w_1-m_1}(-X_2)^{w_2-m_2}\d z.
\end{align*}
Cet élément est bien dans $V_{w_1,w_2}[I_D]$ car pour $\gamma\in\Gamma$:
\begin{equation}
(\gamma,\gamma).Q_h(X_1,X_2)=\int_{\gamma^{-1}\infty}^{\gamma^{-1}0}h(z)(X_1-z)^{w_1}(X_2-z)^{w_2} \d z.
\end{equation}
Donc on peut en déduire que $[(1,1)+(S,S)].Q_h=[(1,1)+(U,U)+(U^2,U^2)].Q_h=0$.\par
En développant de la même manière $P_f(Z)$, on peut identifier les coefficients et on obtient un antécédent dans $V_{w_1,w_2}^{\Q}[I_D]$.
\end{rem}

La proposition \ref{propsection} permet d'obtenir une construction inductive de l'espace $V_{w_1,w_2}^{\Q}[I_D]$ et donc de $E_{w_1,w_2}^{\Q}$.

\begin{prop}\label{prop19}
On a la décomposition :
\begin{equation}
V_{w_1,w_2}^{\Q}[I_D] \cong \bigoplus_{w=|w_1-w_2|}^{w_1+w_2} W_w.
\end{equation}
De plus, la suite suivante est exacte :
\begin{equation}\label{suitexEww}
0\to \Q \to W_{w_1}\otimes 1 \times V_{w_1,w_2}^{\Q}[I_D]\times X_1^{w_1}\otimes W_{w_2}\to E_{w_1,w_2}^{Q}\to 0,
\end{equation}
avec dans l'ordre $1\mapsto \big(1-X^{w_1},1-X_1^{w_1}X_2^{w_2},X^{w_1}(1-X_2^{w_2})\big)$ et $(P_H,P_D,P_V)\mapsto P_H-P_D+P_V$.
\end{prop}

\begin{proof}
La définition de $V_{w_1,w_2}^{\Q}[I_D]$, voir (\ref{defVID}), nous permet de construire la suite exacte de $\Q$-espaces vectoriels:
$$0 \longrightarrow V_{w_1-1,w_2-1}^{\Q}[I_D] \stackrel{\phi}{\longrightarrow} V_{w_1,w_2}^{\Q}[I_D] \stackrel{\psi}{\longrightarrow} W_{w_1+w_2-2}^{\Q} \longrightarrow 0$$
où on définit les $\Gamma$-morphismes par $\phi(P)(X,Y)=(X-Y)P(X,Y)$ et $\psi(P)(Z)=P(Z,Z)$.
L'application $\phi$ est injective et est bien un $\Gamma$-morphisme car $(\gamma X-\gamma Y)(cX+d)(cY+d)=(ad-bc)(X-Y)$ pour tout $\gamma=\mat{a}{b}{c}{d}\in\Gamma$.
L'image de $\phi$ est l'ensemble des polynômes s'annulant diagonalement, c'est à dire le noyau de $\psi$.
Enfin la Proposition \ref{propsection} donne la surjectivité de $\psi$ ainsi que l'existence d'une section.\par

Pour obtenir la suite exacte (\ref{suitexEww}), on observe que la seule formule non triviale parmi les trois espaces $V_{w_1,w_2}^{\Q}[I_H],V_{w_1,w_2}^{\Q}[I_V]$ et $V_{w_1,w_2}^{\Q}[I_D]$ est:
\begin{equation}
1-X_1^{w_1}X_2^{w_2}=\big(1-X^{w_1}\big)+X^{w_1}\left(1-X^{w_2}\right).
\end{equation}
Donc en particulier $\dim E_{w_1,w_2}^{\Q}=\dim V_{w_1,w_2}^{\Q}[I_H]+\dim V_{w_1,w_2}^{\Q}[I_V]+\dim V_{w_1,w_2}^{\Q}[I_D] -1$.
\end{proof}

La démonstration nous fournit un algorithme de calcul de ces espaces pour les petites dimensions:\par
1) Lorsque $w_1=0$ ou $w_2=0$, il n'y a qu'une variable et on a:
$$V_{0,w}^{\Q}[I_D]\cong V_{w,0}^{\Q}[I_D]\cong W_w^{\Q},\quad\text{pour tout } w\geq 0.$$\par
2) Pour tout $w\leq 8$, $W_w^{\Z}=<1-Z^{k-2}>$ donc on a pour $w_1+w_2\leq 8$:
\begin{align*}
V_{2,2}^{\Q}[I_D]&=<1-X_1^2X_2^2,(X_1-X_2)(1-X_1X_2)>,\\
V_{2,4}^{\Q}[I_D]&=<1-X_1^2X_2^4,(X_1-X_2)(1-X_1X_2^3),(X_1-X_2)^2(1-X_2^2)>,\\
V_{2,6}^{\Q}[I_D]&=<1-X_1^2X_2^6,(X_1-X_2)(1-X_1X_2^5),(X_1-X_2)^2(1-X_2^4)>,\\
V_{4,4}^{\Q}[I_D]&=<1-X_1^4X_2^4,(X_1-X_2)(1-X_3X_2^3),(X_1-X_2)^2(1-X_1^2X_2^2),(X_1-X_2)^3(1-X_1X_2)>.
\end{align*}
3) Lorsque $w_1+w_2=10$, on a:
$$W_{10}^{\Q}=<1-X^{10},\quad X^2-3X^4+3X^6-X^8,\quad 4X-25X^3+42X^5-25X^7+4X^9>.$$
Et ainsi on obtient par exemple:
$$V_{2,8}^{\Q}[I_D]=<1-X_1^2X_2^8,(X_1-X_2)(1-X_1X_2^7),(X_1-X_2)(1-X_2^6),P_{2,8}^+(X_1,X_2),P_{2,8}^-(X_1,X_2)>,$$
avec :
\begin{align*}
P_{2,8}^+(X_1,X_2)&=
\left(\frac{28}{45} X_2^2-\frac{70}{70}X_2^4+\frac{28}{70}X_2^6-\frac{1}{45}X_2^8\right)
+2X_1\left(\frac{8}{45}X_2-\frac{56}{70}X_2^3+\frac{56}{70}X_2^5-\frac{8}{45}X_2^7\right)\\
&+X_1^2\left(\frac{1}{45}-\frac{28}{70}X_2^2+\frac{70}{70}X_2^4-\frac{28}{45}X_2^6\right),\\
P_{2,8}^-(X_1,X_2)&=\left(\frac{16}{5} X_2-\frac{35}{3}X_2^3+\frac{28}{3}X_2^5-\frac{5}{3}X_2^7\right)
+2X_1\left(\frac{2}{5}-\frac{35}{6}X_2^2+\frac{35}{3}X_2^4-\frac{35}{6}X_2^6+\frac{2}{5}X_2^8\right)\\
&+X_1^2\left(-\frac{5}{3}X_2+\frac{28}{3}X_2^3-\frac{35}{3}X_2^5+\frac{16}{5}X_2^7\right).
\end{align*}

\subsection{Action de conjugaison sur $W_{w_1,w_2}$}

Notons $\varepsilon=\pm\mat{-1}{0}{0}{1}\in PGL_2(\Z)$. Il agit par conjugaison sur $\Gamma, \H$ et $V_{w}$ par:
\begin{equation}
\varepsilon\matg\varepsilon=\mat{\phantom{-}a}{-b}{-c}{\phantom{-}d},\quad 
\varepsilon.z=-\bar{z}\text{ et }
\varepsilon.P(X) = P(-X).
\end{equation}

Dans le cas classique, on remarque notamment que $\varepsilon\mathcal{I}_1\varepsilon=\mathcal{I}_1$ puis que pour $f\in S_k^{\Q}$, alors :
\begin{equation}\label{modeps}
\omega_f(\varepsilon.z,X)=\omega_f(-\bar{z},X)=\overline{\omega_f(z,-X)}=\varepsilon.\omega_{\bar{f}}(z,X),
\end{equation}
avec $\bar{f}\in \overline{S_{k}^{\Q}}$ une forme antiholomorphe sur $\H$.
Ceci permet d'obtenir le résultat (\ref{Wsom}) dû à Eichler et Shimura.\par

Ces actions s'étendent diagonalement sur $\Gamma^2$ et $\H^2$ pour les trois éléments $(1,\varepsilon), (\varepsilon,1)$ et $(\varepsilon,\varepsilon)$.\par

Commençons par regarder leurs actions sur l'idéal $\mathcal{I}_2$, on a:
\begin{align}
(\varepsilon,\varepsilon)\mathcal{I}_2(\varepsilon,\varepsilon)&=\mathcal{I}_2\label{symI2}\\
\text{et }(1,\varepsilon)\mathcal{I}_2(1,\varepsilon)&=(\varepsilon,1)\mathcal{I}_2(\varepsilon,1).
\end{align}
Notons $\mathcal{I}_2^{-}$ ce second idéal alors:
\begin{multline*}
\mathcal{I}_2^{-}=\Big\langle (1+S,1+S),(S,S)+(S,VS)+(US,VS)+(1,V)-(U^2,V^2),\\
(1+U+U^2,1)((1,1)+(S,S)),(1,1+U+U^2)((1,1)+(S,S))\Big\rangle,
\end{multline*}
avec $V=\varepsilon U\varepsilon = SU^2S = \mat{0}{-1}{1}{1}$.\par

La stabilité de $\H^2$ et donc celle $\tau_2$ par la conjugaison, nous contraint à introduire:
\begin{equation}
W_{w_1,w_2}^{-,\Q} = V_{w_1,w_2}^{\Q}[\mathcal{I}_2^-] = (1,\varepsilon) W_{w_1,w_2}^{\Q}.
\end{equation}
Nous noterons $W_{w_1,w_2}^{-}$ son extension au complexe.\par

Soient $\epsilon_1,\epsilon_2\in\{+,-\}$ des signes. On note $\Per_{w_1,w_2}^{\epsilon_1,\epsilon_2}$ l'espace engendré par les polynômes des bi-périodes $P_{f_1,f_2}$ avec respectivement $f_j$ holomorphe si $\epsilon_j=+$ et antiholomorphe si $\epsilon_j=-$.
La relation (\ref{modeps}) permet d'observer que:
\begin{align}
\Per_{w_1,w_2} &= \Per_{w_1,w_2}^{+,+} = (\varepsilon,\varepsilon) \Per_{w_1,w_2}^{-,-} =\overline{\Per_{w_1,w_2}^{-,-}},\\
(1,\varepsilon)\Per_{w_1,w_2} &= \Per_{w_1,w_2}^{+,-} = (\varepsilon,\varepsilon) \Per_{w_1,w_2}^{-,+} =\overline{\Per_{w_1,w_2}^{-,+}}.
\end{align}

La décomposition de $W_{w_1,w_2}$ se fait alors couplée à celle de $W_{w_1,w_2}^-$.

\begin{prop}
On a la décomposition suivante:
\begin{equation}\label{decSec}
W_{w_1,w_2}+W_{w_1,w_2}^{-}
=\left(E_{w_1,w_2}+(1,\varepsilon)E_{w_1,w_2}\right)\oplus \bigoplus_{(\epsilon_1,\epsilon_2)\in\{+,-\}^2} \Per^{\epsilon_1,\epsilon_2}_{w_1,w_2}.
\end{equation}
\end{prop}

\begin{proof}
Pour démontrer ce résultat, on commence par vérifier l'exactitude de la suite (\ref{suiteex}). Et ainsi, on étend l'application $\Phi_S$ en : 
\begin{equation*}
\varphi_S: W_{w_1,w_2}+W_{w_1,w_2}^-\to W_{w_1}/E_{w_1}\otimes W_{w_2}/E_{w_2}, 
P\mapsto \left[(1,1)+(S,S)\right] P.
\end{equation*}

L'application est bien définie car la détermination de $\mathcal{I}_2$ montre que l'image par $(1,1)+(S,S)$ de $V_{w_1,w_2}^{\Q}[\mathcal{I}_2]$ est :
$$V_{w_1,w_2}[(1+S,1),(1+U+U^2,1),(1,1+S),(1,1+U+U^2)]=W_{w_1}\otimes W_{w_2}.$$

On montre alors que $\varphi_S$ est surjective. 
Soit $P_1\otimes P_2\in W_{w_1}\otimes W_{w_2}$. 
Alors d'après le Théorème d'Eichler-Shimura la classe de $P_j$ dans $W_{w_j}/E_{w_j}$ admet un unique antécédent $\omega_j\in\Omega_{k_j}^+\oplus \Omega_{k_j}^-$ pour $j=1,2$. Dans ces notations, les signes précisent le caractère holomorphe ou antiholomorphe des $1$-formes invariantes à valeurs dans $V_{w_j}$.\par

Posons $P=\langle \omega_1\wedge\omega_2, \tau_2\rangle$. On sait que $P\in W_{w_1,w_2}+W_{w_1,w_2}^-$ et de plus :
\begin{equation*}
\varphi_S(P) = \left[(1,1)+(S,S)\right].\langle\omega_1\wedge\omega_2, \tau_2\rangle
=\langle\omega_1\wedge\omega_2, \left[(1,1)+(S,S)\right].\tau_2\rangle
=\langle\omega_1\wedge\omega_2, \tau_1\times\tau_1\rangle = P_1\otimes P_2.
\end{equation*}

Ceci nous donne une section de $\varphi_S$ uniquement valable sur le corps des complexes. On remarque de plus que l'image de cette section est bien $\bigoplus_{(\epsilon_1,\epsilon_2)\in\{+,-\}^2} \Per^{\epsilon_1,\epsilon_2}_{w_1,w_2}$, les polynômes des bi-périodes de formes modulaires harmoniques. Cette section nous permet d'obtenir la décomposition (\ref{decSec}) sous réserve de démontrer que :
\begin{equation}
\Ker \varphi_S = E_{w_1,w_2}+(1,\varepsilon)E_{w_1,w_2}.
\end{equation}

Le calcul du noyau reste valide sur $\Q$. On commence par regarder les éléments de $W_{w_1,w_2}^{\Q}$ annulés par $(1,1)+(S,S)$ dans $W_{w_1}^{\Q}\otimes W_{w_2}^{\Q}$. Ce sont les éléments annulés par l'idéal:
$$\mathcal{I}_2+[(1,1)+(S,S)]\Z[\Gamma^2]=\left[(1,1)+(S,S);(1,1)+(U,U)+(U^2,U^2)\right]\Z[\Gamma^2]=I_D.$$
Ceci démontre que $\{P\in W_{w_1,w_2}^{\Q}\text{ tel que }P|_{(1,1)+(S,S)}=0\}=V_{w_1,w_2}^{\Q}[I_D].$\par

On en déduit que $\{P\in W_{w_1,w_2}^{-,\Q}\text{ tel que }P|_{(1,1)+(S,S)}=0\}=(1,\varepsilon)V_{w_1,w_2}^{\Q}[I_D].$

Il nous reste à déterminer les éléments qui s'envoient sur $W_{w_1}^{\Q}\otimes E_{w_2}^{\Q}+E_{w_1}^{\Q}\otimes W_{w_2}^{\Q}$. 

Or on obtient, pour $P\in V_{w_1,w_2}^{\Q}$ la série d'équivalence:
\begin{align*}
P&\in V_{w_1,w_2}^{\Q}[I_H]\Leftrightarrow (1+S,1)P=(1+U+U^2,1)P=(1,1-T)P=0\\
&\Leftrightarrow P|_{(1,1)+(S,S)}=(1,1-S)P\text{ et }(1+S,1)P=(1+U+U^2,1)P=(1,1-T)P=0\\
&\Leftrightarrow P|_{(1,1)+(S,S)}\in (1,1-S)\left(V_{w_1}^{\Q}[1+S,1+U+U^2]\otimes V_{w_2}^{\Q}[1-T]\right)\\
&\Leftrightarrow P|_{(1,1)+(S,S)}\in W_{w_1}^{\Q}\otimes E_{w_2}^{\Q}.
\end{align*}
De plus, on remarque que $(1,\varepsilon)I_H(1,\varepsilon)=I_H$. Donc il n'y a pas de nouveau terme pour la conjugaison.
Ceci se symétrise pour $I_V$ sans difficulté car on a aussi $E_{w_1}^{\Q}=(1-S)V_{w_1}^{\Q}[1-US]$.\par 

On obtient ainsi:
$$\Ker (\varphi_S^{\Q})=V_{w_1,w_2}^{\Q}[I_H]+ V_{w_1,w_2}^{\Q}[I_V]+ V_{w_1,w_2}^{\Q}[I_D]+ (1,\varepsilon) V_{w_1,w_2}^{\Q}[I_D] = E_{w_1,w_2}^{\Q}+(1,\varepsilon).E_{w_1,w_2}^{\Q}.$$\par

En prenant la restriction à l'espace $W_{w_1,w_2}^{\Q}$, on obtient bien que: $\Ker (\Phi_S^{\Q})=E_{w_1,w_2}^{\Q}$. Toutefois l'application n'est pas toujours surjective.
\end{proof}

\subsection{Calcul de l'écart entre $W_{w_1,w_2}$ et $E_{w_1,w_2}$}

La relation (\ref{symI2}) permet d'obtenir que $W_{w_1,w_2}^{\Q}$ est stable par $(\varepsilon,\varepsilon)$. Ainsi on peut considérer les parties paires et impaires par:
\begin{align}
P^+(X_1,X_2)&=\frac{1}{2}\left(P(X_1,X_2)+P(-X_1,-X_2)\right)=\sum_{m_1+m_2\text{ pair}}A_{m_1,m_2}X_1^{m_1}X_2^{m_2},\\
P^-(X_1,X_2)&=\frac{1}{2}\left(P(X_1,X_2)-P(-X_1,-X_2)\right)=\sum_{m_1+m_2\text{ impair}}A_{m_1,m_2}X_1^{m_1}X_2^{m_2}.
\end{align}
Si $P\in W_{w_1,w_2}^{\Q}$ alors $P^+$ et $P^-$ sont aussi des éléments de $W_{w_1,w_2}^{\Q}$.
Cette stabilité par conjugaison double permet de définir les espaces:
\begin{align}
W_{w_1,w_2}^{pair,\Q}&=\{P^+(X_1,X_2)\text{ pour }P\in W_{w_1,w_2}^{\Q}\},\\
\text{et }W_{w_1,w_2}^{imp,\Q}&=\{P^-(X_1,X_2)\text{ pour }P\in W_{w_1,w_2}^{\Q}\}.
\end{align}

Ces ensembles, et plus simplement encore leur dimension, peuvent être compilés car ils sont le noyau des quatre relations de Manin double.

De plus, on obtient la même stabilité pour les idéaux $I_D, I_H$ et $I_V$ donc on peut construire $E_{w_1,w_2}^{pair,\Q}$ et $E_{w_1,w_2}^{imp,\Q}$. 
La dimension de $E_{w_1,w_2}^{\Q}$ donnée par la Proposition \ref{prop19} peut se scinder simplement en parties paire et impaire par :
\begin{align}
V_{w_1,w_2}^{pair,\Q}[I_D]&\cong \bigoplus_{w=|w_1-w_2|}^{w_1+w_2} W_w^{i^{w+w_1+w_2}},\\
\text{et }V_{w_1,w_2}^{imp,\Q}[I_D]&\cong \bigoplus_{w=|w_1-w_2|}^{w_1+w_2} W_w^{-i^{w+w_1+w_2}}.
\end{align}
Lorsqu'un des poids est inférieur à $8$ ou vaut $12$ alors l'espace d'arrivée est nulle et on a $W_{w_1,w_2}=E_{w_1,w_2}$.
On résume dans le tableau les dimensions que l'on a pu calculer:\\

\begin{tabular}{cc|ccc|ccc}
$w_1$ & $w_2$ & $\dim W_{w_1,w_2}^{pair}$ & $\dim E_{w_1,w_2}^{pair}$ & écart & $\dim W_{w_1,w_2}^{imp}$ & $\dim E_{w_1,w_2}^{imp}$ & écart\\
\hline
$10$ & $10$ & $15$ & $13$ & $2$ & $13$ & $12$ & $1$ \\
$14$ & $14$ & $24$ & $22$ & $2$ & $22$ & $21$ & $1$ \\
$16$ & $16$ & $29$ & $27$ & $2$ & $27$ & $26$ & $1$ \\
$18$ & $18$ & $35$ & $33$ & $2$ & $33$ & $32$ & $1$ \\
$20$ & $20$ & $42$ & $40$ & $2$ & $40$ & $39$ & $1$ \\
$24$ & $24$ & $57$ & $55$ & $2$ & $55$ & $54$ & $1$ \\
\hline

$10$ & $14$ & $19$ & $17$ & $2$ & $17$ & $15$ & $2$ \\
$10$ & $16$ & $21$ & $19$ & $2$ & $19$ & $17$ & $2$ \\
$10$ & $18$ & $23$ & $21$ & $2$ & $21$ & $19$ & $2$ \\
$10$ & $20$ & $25$ & $23$ & $2$ & $23$ & $21$ & $2$ \\
$10$ & $24$ & $29$ & $27$ & $2$ & $27$ & $25$ & $2$ \\

$14$ & $16$ & $27$ & $25$ & $2$ & $25$ & $23$ & $2$ \\
$14$ & $18$ & $29$ & $27$ & $2$ & $27$ & $25$ & $2$ \\
$14$ & $20$ & $32$ & $30$ & $2$ & $30$ & $28$ & $2$ \\
$14$ & $24$ & $37$ & $35$ & $2$ & $35$ & $33$ & $2$ \\

$16$ & $18$ & $33$ & $31$ & $2$ & $31$ & $29$ & $2$ \\
$16$ & $20$ & $35$ & $33$ & $2$ & $33$ & $31$ & $2$ \\
$16$ & $24$ & $41$ & $39$ & $2$ & $39$ & $37$ & $2$ \\

$18$ & $20$ & $39$ & $37$ & $2$ & $37$ & $35$ & $2$ \\
$18$ & $24$ & $45$ & $43$ & $2$ & $43$ & $41$ & $2$ \\
$20$ & $24$ & $49$ & $47$ & $2$ & $47$ & $45$ & $2$ \\

\hline

$10$ & $22$ & $29$ & $25$ & $4$ & $27$ & $23$ & $4$\\
$14$ & $22$ & $37$ & $33$ & $4$ & $35$ & $31$ & $4$\\
$16$ & $22$ & $41$ & $37$ & $4$ & $39$ & $35$ & $4$\\
$18$ & $22$ & $45$ & $41$ & $4$ & $43$ & $39$ & $4$\\

$20$ & $22$ & $49$ & $45$ & $4$ & $47$ & $43$ & $4$\\
$22$ & $22$ & $57$ & $49$ & $8$ & $55$ & $48$ & $7$\\
$34$ & $34$ & $127$ & $109$ & $18$ & $125$ & $108$ & $17$\\

\end{tabular}
\\

La différence entre $W_{w_1,w_2}$ et $E_{w_1,w_2}$ donne les cas de surjectivité de l'application $\Phi_S$ que l'on peut restreindre aux parties paires et impaires.
Définissons :
\begin{align}
\Phi_S^{pair}: W_{w,w}^{pair}&\to \left(W_{w}/E_{w}\otimes W_{w}/E_{w}\right)\cap V_{w,w}^{pair},\\
\text{et }\Phi_S^{imp}: W_{w,w}^{imp}&\to \left(W_{w}/E_{w}\otimes W_{w}/E_{w}\right)\cap V_{w,w}^{imp}.
\end{align}
 Les espaces d'arrivée sont respectivement les parties globales paires et impaires de $W_{w_1}/E_{w_1}\otimes W_{w_2}/E_{w_2}$. 
Ils sont donc chacun de dimension $2\dim S_{k_1}\dim S_{k_2}$.\par

Au vue de la suite exacte (\ref{suiteex}), on observe donc dans la table que $\Phi_S^{pair}$ est toujours surjective. 
Alors que $\Phi_S^{imp}$ est surjective sauf lorsque $w_1=w_2$ et $S_{w_1+2}\neq 0$. 
Nous ne donnons pas de démonstration générale de la surjectivité à part celle observée dans la table.

\subsection{Démonstration du Théorème \ref{thm2}}

La surjectivité de l'application peut se traduire par des résultats d'irrationalité formulés dans le Théorème \ref{thm2}.\par
On rappelle que pour $f_{w_j}\in S_{w_j+2}$ une forme propre de Hecke, on écrit :
\begin{equation*}
P_{f_{w_j}}=\Omega_{w_j}^+P_{w_j}^++i\Omega_{w_j}^-P_{w_j}^-\text{ et }r_{w_j}=\Omega_{w_j}^+/\Omega_{w_j}^-.
\end{equation*}
On part alors des expressions:
\begin{align}
\Phi_S^{pair}(P_{f_{w_1},f_{w_2}}^+) &= \Omega_{w_1}^+\Omega_{w_2}^+ (P_{w_1}^+\otimes P_{w_2}^+) - \Omega_{w_1}^-\Omega_{w_2}^- (P_{w_1}^-\otimes P_{w_2}^-),\\
\text{et }\Phi_S^{imp}(P_{f_{w_1},f_{w_2}}^-) &= i\Omega_{w_1}^+\Omega_{w_2}^- (P_{w_1}^+\otimes P_{w_2}^-) + i\Omega_{w_1}^-\Omega_{w_2}^+ (P_{w_1}^-\otimes P_{w_2}^+).
\end{align}
On suppose par l'absurde $r_{w_1}r_{w_2}\in\Q$. On sait que $P_{w_1}^+\otimes P_{w_2}^+$ et $P_{w_1}^-\otimes P_{w_2}^-$ sont à coefficients rationnels et dans l'image de $\Phi_S$. Donc il existe $P_1, P_2\in W_{w_1,w_2}^{\Q,pair}$ tel que : $\Phi_S(P_1)=P_{w_1}^+\otimes P_{w_2}^+$ et $\Phi_S(P_2)=P_{w_1}^-\otimes P_{w_2}^-$. Posons alors $P=r_{w_1}r_{w_2}P_1-P_2\in W_{w_1,w_2}^{pair,\Q}$ de manière à avoir $\Phi_S(P) = \Omega_{w_1}^-\Omega_{w_2}^-\Phi_S(P_{f_{w_1},f_{w_2}}^+)$.
On en déduit que :
$$P_{f_{w_1},f_{w_2}}^+ \in \C P+\Ker (\Phi_S) =\C P +E_{w_1,w_2}.$$ 
Donc le $\Q$-espace $\Q P+E_{w_1,w_2}^{pair,\Q}$ est strictement contenu dans $W_{w_1,w_2}^{pair,\Q}$. Pourtant son extension au corps des complexes contient $\Per_{w_1,w_2}^{pair}$. 
Quitte à compléter par la partie impaire, ceci contre-dit la minimalité de $W_{w_1,w_2}^{\Q}$ donnée par le Théorème \ref{thm1}.\par

De même, si l'on suppose $w_1\neq w_2$ et $r_{w_1}/r_{w_2}\in\Q$ alors la surjectivité permet de construire un antécédent rationnel $P\in W_{w_1,w_2}^{\Q}$ tel que :
\begin{equation}
\Phi_S(P)=\frac{r_{w_1}}{r_{w_2}}(P_{w_1}^+\otimes P_{w_2}^-) + (P_{w_1}^-\otimes P_{w_2}^+)=i\Omega_{w_1}^-\Omega_{w_2}^+\Phi_S(P_{f_{w_1},f_{w_2}}^-).
\end{equation}
Le $\Q$-espace $\Q P+E_{w_1,w_2}^{imp,\Q}$ est strictement contenu dans $W_{w_1,w_2}^{imp,\Q}$ et vérifie $P_{f_{w_1},f_{w_2}}^-\in \C P+\Ker (\Phi_S)$. Ceci contredit à nouveau le Théorème \ref{thm1}.

\begin{rem}
On remarquera que le même raisonnement dans le cas du polynôme des périodes d'une seule forme donne le résultat trivial $\Omega_w^+/(i\Omega_w^-)=-ir_w\notin\Q$.
Et ainsi le Théorème \ref{thm2} fournit une motivation particulière à l'étude des polynôme des bi-périodes et à leurs généralisations pour un nombre quelconque de formes modulaires.
\end{rem}

Le cas de non-surjectivité est naturel. On peut le voir comme conséquence du Théorème \ref{thm1} et on obtient le résultat général:

\begin{prop}
Soit $w$ un poids tel que $S_{w+2}\neq 0$ alors $\Phi_S^{imp}$ n'est pas surjective.
\end{prop}

\begin{proof}
Pour toute forme $f\in S_{w+2}$ propre pour les opérateurs de Hecke. On a :
\begin{equation}
\Phi_S(P_{f,f})= P_{f} \otimes P_{f} = 
\left(\Omega_f^+\right)^2 (P_f^+\otimes P_f^+) 
-\left(\Omega_f^-\right)^2 (P_f^-\otimes P_f^-)
+i\Omega_f^+\Omega_f^- (P_f^+\otimes P_f^- + P_f^-\otimes P_f^+).
\end{equation}
Donc $\Phi_S^{imp}(P_{f,f}^-)=i\Omega_f^+\Omega_f^-(P_f^+\otimes P_f^- + P_f^-\otimes P_f^+)$.
Or si l'on suppose $\Phi_S^{imp}$ surjective alors il existe $P_1,P_2\in W_{w,w}^{imp}$ tels que $\Phi_S^{imp}(P_1)=P_f^+\otimes P_f^-$ et $\Phi_S^{imp}(P_2)=P_f^+\otimes P_f^-$. 
On a $P_1$ et $P_2$ linéairement indépendant et :
\begin{equation*}
\Phi_S^{imp}(P_1+P_2)=P_f^+\otimes P_f^- + P_f^-\otimes P_f^+=\Phi_S^{imp}(P_{f,f}^-)/\left(i\Omega_f^+\Omega_f^-\right).
\end{equation*} 
Donc en remplaçant $P_1$ et $P_2$ par $P_1+P_2$ dans une famille génératrice sur $\Q$ de $W_{w,w}^{\Q}$, on obtiendrait un espace strictement plus petit et vérifiant le Théorème \ref{thm1}. Ce qui est absurde.
\end{proof}

\bibliography{bibmod}

\bigskip
\textbf{Contact E-mail:} \textsf{provostmath@gmail.com}
\end{document}